\documentclass[11pt]{article}
\usepackage{amssymb, amsfonts, latexsym, amsthm, amsmath, verbatim, framed,cite,mathrsfs}
\usepackage{graphicx}
\usepackage{graphics}
\usepackage{calligra}
\usepackage{url}
\usepackage[margin=1in]{geometry}
\usepackage{tikz}
\usepackage{tabmac}
\usepackage{overpic}
\usepackage{titling}
\usepackage{mathtools}

\newcommand{\CC}{\mathbb{C}}
\newcommand{\RR}{\mathbb{R}}
\newcommand{\NN}{\mathbb{N}}
\newcommand{\ZZ}{\mathbb{Z}}

\newcommand{\anTLn}{n\widehat{\text{TL}_n}}
\newcommand{\kk}{m}
\newcommand{\anTLkn}{n\widehat{\text{TL}_{\kk n}}}
\newcommand{\inv}{\text{Inv}}
\newcommand{\len}{\ell}
\newcommand{\LL}{a}
\newcommand{\kconj}{k}
\newcommand{\bin}{\operatorname{b}}
\newcommand{\iso}{\chi}
\newcommand{\sd}{\iota}
\newcommand{\pp}{a}
\newcommand{\qq}{b}
\newcommand{\rr}{c}
\newcommand{\mm}{d}

\theoremstyle{plain}

\newtheorem{claim}{Claim}[section]
\newtheorem{lemma}[claim]{Lemma}

\newtheorem{thm}[claim]{Theorem}
\newtheorem{prop}[claim]{Proposition}

\newtheorem{cor}[claim]{Corollary}

\theoremstyle{definition}
\newtheorem{Def}[claim]{Definition}
\newtheorem{rmk}[claim]{Remark}
\newtheorem{ex}[claim]{Example}
\newtheorem*{exnonum}{Example}

\title
{Applying Parabolic Peterson: Affine Algebras and the Quantum Cohomology of the Grassmannian}

\makeatletter
\renewcommand\@date{{%
		\vspace{-\baselineskip}%
		\large\centering
		\begin{tabular}{@{}c@{}}
			Jonathan Cookmeyer\\
			University of California\\
			366 LeConte Hall MC 7300\\
			Berkeley, CA 94720\\
			jcookmeyer@berkeley.edu
		\end{tabular}%
		\qquad\qquad
		\begin{tabular}{@{}c@{}}
			Elizabeth Mili\'cevi\'c\\
			Haverford College\\
			370 Lancaster Avenue\\
			Haverford, PA 19041\\
			emilicevic@haverford.edu
		\end{tabular}

}}
\makeatother

\providecommand{\keywords}[1]{\textbf{Keywords.} #1}
\providecommand{\ack}[1]{\textbf{Funding.} #1}
\providecommand{\subjclass}[1]{\textbf{2010 MSC Codes.} #1}

\begin{document}

\maketitle

\begin{abstract}{
	The Peterson isomorphism relates the homology of the affine Grassmannian to the quantum cohomology of any flag variety.  In the case of a partial flag, Peterson's map is only a surjection, and one needs to quotient by a suitable ideal on the affine side to map isomorphically onto the quantum cohomology.  We provide a detailed exposition of this parabolic Peterson isomorphism in the case of the Grassmannian of $m$-planes in complex $n$-space, including an explicit recipe for doing quantum Schubert calculus in terms of the appropriate subset of non-commutative $k$-Schur functions.  As an application, we recast Postnikov's affine approach to the quantum cohomology of the Grassmannian as a consequence of parabolic Peterson by showing that the affine nilTemperley-Lieb algebra arises naturally when forming the requisite quotient of the homology of the affine Grassmannian.}
\end{abstract}


\tableofcontents

\vskip 12pt

{

\noindent
\\

\noindent
\\

\noindent
\ack{The second author was partially supported by NSF grant DMS--1600982.}

}

\section{Introduction}

The theory of quantum cohomology was developed in the early 1990s by physicists working in the field of superstring theory, who were able to provide a partial answer to the Clemens conjecture counting the number of rational curves of a given degree on a general quintic threefold. A mathematical formulation of quantum cohomology pioneered by Givental formally proved the answer proposed by the physicists \cite{Givental}, and simultaneously introduced the Gromov-Witten invariants, which encode all such statistics in enumerative geometry, in joint work with Kim \cite{GiventalKim}. The (small) quantum cohomology ring is a deformation of the classical cohomology by a sequence of quantum parameters $q$, and the Schubert basis elements then form a $\ZZ[q]$-basis for the quantum cohomology.  The theory of quantum Schubert calculus seeks non-recursive, positive combinatorial formulas for expressing the quantum product of two Schubert classes in terms of the Schubert basis.

The Grassmannian has emerged as a favorite test case among all homogeneous spaces for carrying out the quantum Schubert calculus program.  In this paper, we specialize to the case of the type $A$ Grassmannian $Gr(m,n)$, where the Schubert classes are indexed by minimal length coset representatives in $S_n/(S_m\times S_{n-m})$, or equivalently by Young diagrams fitting inside an $m\times (n-m)$ rectangle.  The product of two quantum Schubert classes is a $\ZZ_{\geq 0}[q]$-linear combination of the Schubert basis elements, and the associated quantum Littlewood-Richardson coefficients have been well-studied. Already by the late 1990s, various recursive and signed rules for computing quantum Schubert products appeared, such as the quantum Pieri rule of Bertram \cite{Bertram} and the rim hook rule of Bertram, Ciocan-Fontanine, and Fulton \cite{BCFF}. A positive combinatorial formula for all quantum Littlewood-Richardson coefficients has now been proved by Buch, Kresch, Purbhoo, and Tamvakis via counting puzzles in a related two-step flag variety \cite{BKPT}.

In a series of lectures delivered at MIT in the late 1990s, Peterson developed a deep connection between the equivariant quantum cohomology of any partial flag variety $G/P$ and the equivariant homology of the affine Grassmannian $\mathcal{G}r_G$ \cite{Pet}.  This result, later proved by Lam and Shimozono \cite{LS10}, says that there is a surjective homomorphism of Hopf algebras 
\[ H_*^T(\mathcal{G}r_G) \twoheadrightarrow QH^*_T(G/P),\]
up to localization.  In the case of the complete flag variety $G/B$, the map above is in fact an isomorphism, and otherwise, one needs to mod out by a suitable ideal $J_P$ on the affine side in order to map injectively onto the smaller ring $QH^*_T(G/P)$.  Somewhat surprisingly, Peterson's isomorphism also admits a non-equivariant shadow, and it is one primary goal of this paper to provide a detailed account of the parabolic Peterson isomorphism for the type $A$ Grassmannian, which says that after a suitable localization,
\[ H_*(\mathcal{G}r_{SL_n})/J_P \cong QH^*(Gr(m,n)).\]

In the type $A$ setting, the Schubert classes in the homology of the affine Grassmannian were shown by Lam to be represented by the $k$-Schur functions of Lapointe and Morse \cite{LMkSchur}.  The approach in \cite{LamkSchur} proves that the non-commutative $k$-Schur functions defined in \cite{LamAJM} span the affine Fomin-Stanley subalgebra $\mathbb B$ of the nilHecke ring of Kostant and Kumar \cite{KK}, and in turn that $H_*(\mathcal{G}r_{SL_n}) \cong \mathbb B$.  The non-commutative $k$-Schur functions are indexed by minimal length elements in the quotient of the affine symmetric group $\tilde S_n/S_n$.  Any such $w \in \tilde S_n/S_n$ corresponds bijectively to a $k$-bounded partition $\lambda(w)$ having first part no larger than $k=n-1$.  Affine Schubert calculus, the fundamental goal of which is to combinatorially describe the multiplication of two affine Schubert classes, as modeled by the non-commutative $k$-Schur functions ${\bf s}_{\lambda(w)}^{(k)}$ or some other variation, has been a booming mathematical enterprise for well over a decade; see the book \cite{kschur} for a comprehensive survey of this literature.

Despite the splash the Peterson isomorphism originally made by bridging the areas of affine and quantum (co)homology, much of the subsequent Schubert calculus literature has remained squarely \emph{either} affine \emph{or} quantum, with comparatively fewer results utilizing the precise connection between the two.  Even in cases where the Peterson isomorphism plays a prominent role, such as in Lam and Shimozono's rational substitution from quantum Schubert polynomials to $k$-Schur functions via the Toda lattice \cite{LSToda}, or Kirillov and Maeno's interpretation of quantum Schubert calculus in terms of a family of braided differential operators \cite{KMbraid}, the focus is often on the case of the complete flag variety, for which the Peterson isomorphism is simplest to state.  Beyond its original appearance in \cite{LS10}, the parabolic Peterson isomorphism plays a more selective role; see \cite{LNSSS}, \cite{HeLam}, and \cite{CMPparpet} for several works which do prominently directly feature the parabolic version of the ``quantum equals affine'' phenomenon.

\subsection{Statement of the main results}

In this paper, we provide a concrete recipe for doing quantum Schubert calculus in the type $A$ Grassmannian in terms of the $j$-basis of non-commutative $k$-Schur functions in a quotient of the affine Fomin-Stanley algebra.  We also present examples of results in the literature on the quantum cohomology of the Grassmannian which can then be viewed as direct consequences of the parabolic Peterson isomorphism.

The first part of this paper is devoted to explaining how to perform calculations in the ring $QH^*(Gr(m,n))$ using the parabolic Peterson isomorphism, which we suitably rephrase in Theorem \ref{mainintro} for this purpose.  Because the quantum cohomology of the Grassmannian is especially well-understood, our hope is that this paper can serve as a self-contained introduction to the Peterson isomorphism for readers less familiar with the language of affine flag varieties, affine Weyl groups, affine root systems, and general reductive algebraic groups.  We provide a comprehensive user's manual for the case of the Grassmannian as a means to pave the way for similar applications to other flag varieties, in which many interesting questions in quantum Schubert calculus remain open.  The following statement combines Propositions \ref{braid} and \ref{piP} with Theorem \ref{GrPeterson}.

\begin{thm}\label{mainintro}
	Let $P$ be the maximal parabolic subgroup such that $SL_n(\CC)/P \cong Gr(m,n)$, and let $v$ be any minimal length coset representative in $S_n/(S_m \times S_{n-m})$.  The Peterson isomorphism gives a correspondence between (representatives of) non-commutative $k$-Schur functions and quantum Schubert classes:
	\begin{equation}
	\begin{aligned}
	\Psi_P : ( H_*(\mathcal{G}r_{SL_n})/ J_P)[ ({\bf \overline{s}}_{\lambda(u)}^{(k)})^{-1}]& \longrightarrow QH^*(Gr(\kk,n))[q^{-1} ] \\
	{\bf \overline{s}}_{\lambda(v u^r)}^{(k)} &\longmapsto q^{-r} \sigma_v,
	\end{aligned}
	\end{equation}
	where 
	\begin{equation}
	u = s_{\kk,\kk+1,...,n-1,\kk-1,...,1,0}
	\end{equation}
	is an element of the affine symmetric group $\tilde{S}_n$ whose corresponding $k$-bounded partition $\lambda(u)$ is a hook shape.  Here, $r$ is any positive integer such that $vu^r$ is a minimal length coset representative in $\tilde S_n/S_n$.  Moreover, if $w \in \tilde S_n/S_n$ supports a braid relation, then ${\bf s}_{\lambda(w)}^{(k)} \in J_P$ so that ${\bf \overline{s}}_{\lambda(w)}^{(k)} = 0$.\end{thm}

We illustrate the ideal $J_P$ corresponding to $Gr(1,3)$ in Figure \ref{Fig:J_PExIntro}. See Section \ref{Sec:review} for the definitions of the terms appearing in the statement of Theorem \ref{mainintro}, the proof of which is found in Section \ref{PetIso}.

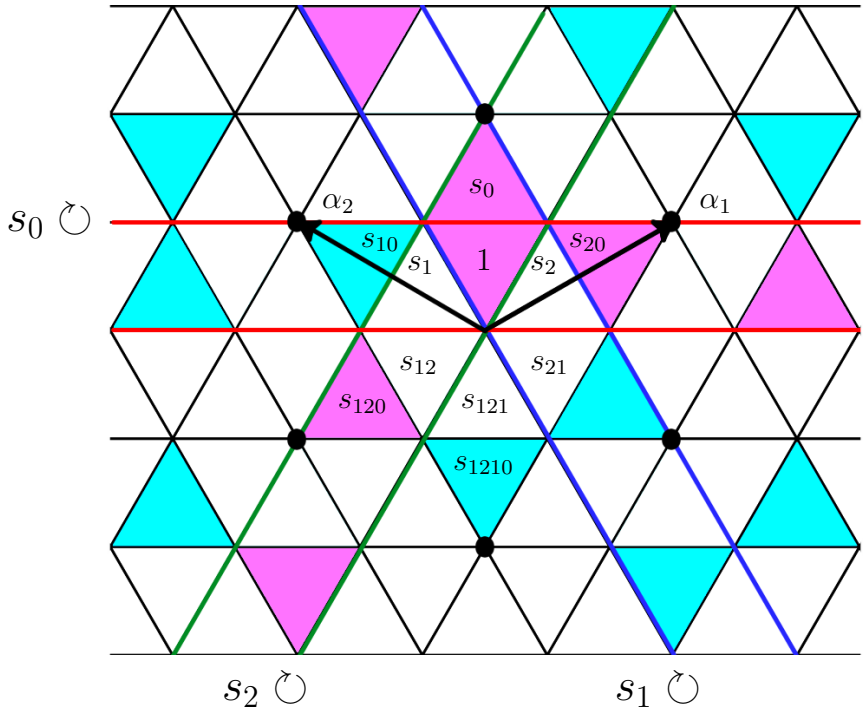
\begin{figure}[h]
	\begin{centering}
\begin{overpic}[width=0.62\linewidth]{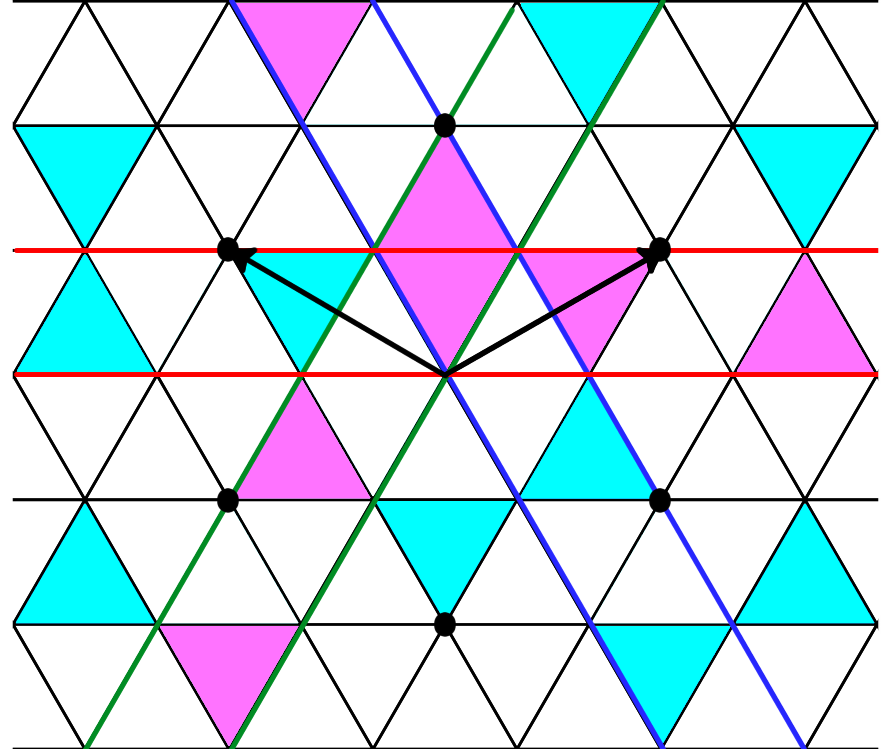}
	\put(49,50){\bf \large $1$}
	\put(40,50){\bf \large $s_1$}
	\put(56,50){\bf \large $s_2$}
	\put(39,37){\bf \large $s_{12}$}
	\put(47,32){\bf \large $s_{121}$}
	\put(46,24){\bf \large $s_{1210}$}
	\put(56,37){\bf \large $s_{21}$}
	\put(29,58){\bf \large $ \alpha_2$}
	\put(78,58){\bf \large $ \alpha_1$}
	\put(48,60){\bf \large $ s_0$}
	\put(-12,55){\bf \LARGE $s_0  \circlearrowright$}
	\put(67,-6){\bf \LARGE $ s_1  \circlearrowright$}
	\put(16,-6){\bf \LARGE $ s_2  \circlearrowright$}
	\put(34,53){\bf \large $ s_{10}$}
	\put(61,53){\bf \large $ s_{20}$}
	\put(31,32){\bf \large $ s_{120}$}
\end{overpic}
\vskip 30pt
\caption{The elements of $J_P$ for the parabolic $P$ corresponding to $Gr(1,3)$ are colored blue. The minimal length coset representatives which do not lie in $J_P$ are colored pink.}\label{Fig:J_PExIntro} 
\end{centering}
\end{figure}

 We now demonstrate how to use Theorem \ref{mainintro} by computing a product in $QH^*(Gr(1,3))$ for illustration purposes; this is an abbreviated version of Example \ref{PetEx} in the body of the paper, where all the relevant terminology has been properly defined.

\begin{exnonum} We use the parabolic Peterson isomorphism to calculate the product $\sigma_{\tiny \tableau[Ys]{& \ }}*\sigma_{\tiny \tableau[Ys]{ & \ }}$ in $QH^*(Gr(1,3))$.  First, note that
\begin{equation*}
\overline{\bf s}_{\tiny \tableau[Ys]{\ }}^{(k)} = \overline{\bf s}^{(k)}_{\lambda(s_{0})}=\overline{\bf s}^{(k)}_{\lambda(s_{21} \cdot s_{120})} \longmapsto q^{-1}\sigma_{\tiny \tableau[Ys]{& \ }}.
\end{equation*}
Therefore, we should compute the following product of non-commutative $k$-Schur functions:
\begin{equation*}
\begin{aligned}
{\bf s}_{\tiny \tableau[Ys]{\ }}^{(k)} \cdot {\bf s}_{\tiny \tableau[Ys]{\ }}^{(k)} &= (A_{0} + A_1 + A_2)^2 \\&=
A_{01}+A_{02}+A_{10} + A_{12}+ A_{20} + A_{21}  \\&= 
{\bf s}^{(k)}_{\lambda(s_{10})} + 
{\bf s}^{(k)}_{\lambda(s_{20})} \\
\implies \overline{\bf s}_{\tiny \tableau[Ys]{\ }}^{(k)} \cdot \overline{\bf s}_{\tiny \tableau[Ys]{\ }}^{(k)} & \equiv \overline{\bf s}_{\lambda(s_{20})}^{(k)} \operatorname{mod} J_P \ = \overline{\bf s}_{\tiny \tableau[Ys]{\\ \ }}^{(k)}.
\end{aligned}
\end{equation*}
The pink alcoves in Figure \ref{Fig:J_PExIntro} are the elements of $\tilde S_3/S_3$ which lie outside of $J_P$ and can thus be used to calculate in $QH^*(Gr(1,3))$.  Figure \ref{Fig:J_PExIntro} also shows those elements of $\tilde S_3/S_3$ which are in $J_P$ in blue; note that one of the two terms in the expansion of this product, namely ${\bf s}_{\lambda(s_{10})}^{(k)} \in J_P$, does not survive the quotient in the parabolic Peterson.  

Finally, we need only map this product of $k$-Schurs under the Peterson isomorphism:
\begin{equation*}
\overline{\bf s}^{(k)}_{\tiny \tableau[Ys]{\\ \ }} =\overline{\bf s}^{(k)}_{\lambda(s_{20})}= \overline{\bf s}^{(k)}_{\lambda(s_1 \cdot s_{120})} \longmapsto q^{-1}\sigma_{\tiny \tableau[Ys]{ \ }}.
\end{equation*}
Putting these results together, we have shown that
\begin{equation*}
q^{-1}\sigma_{\tiny \tableau[Ys]{& \ } }* q^{-1}\sigma_{\tiny \tableau[Ys]{& \ }} \longmapsto \overline{\bf s}_{\tiny \tableau[Ys]{\ }}^{(k)} \cdot \overline{\bf s}_{\tiny \tableau[Ys]{\ }}^{(k)}\equiv \overline{\bf s}_{\tiny \tableau[Ys]{\\ \ }}^{(k)}\longmapsto q^{-1} \sigma_{\tiny \tableau[Ys]{\ }}.
\end{equation*}
Multiplying by $q^2$ to clear denominators, we have calculated that in $QH^*(Gr(1,3)),$
\begin{equation*} \sigma_{\tiny \tableau[Ys]{& \ }}*\sigma_{\tiny \tableau[Ys]{ & \ }}  = q \sigma_{\tiny \tableau[Ys]{\ }}.
\end{equation*}
\end{exnonum}

Our primary motivation in providing a careful analysis of the parabolic Peterson isomorphism, however, was not simply to reproduce formulas for Schubert products in $QH^*(Gr(m,n))$.  We further aim to demonstrate that many existing results about the quantum cohomology of the Grassmannian which were discovered independently can actually be reinterpreted as corollaries of the parabolic Peterson isomorphism.  To provide a concrete example, our primary application in Theorem \ref{PostIntro} recasts Postnikov's affine approach to quantum Schubert calculus from \cite{P05} as a consequence of Theorem \ref{mainintro}.  

Inspired by the interpretation of a localization of the quantum cohomology of the Grassmannian as a representation of the affine Lie algebra $\widehat{\mathfrak{sl}}_n$, Postnikov proves in \cite{P05} that a localized subalgebra of the affine nilTemperley-Lieb algebra, generated by certain distinguished elements ${\bf \tilde e}^n_r$ and ${\bf \tilde h}^n_r$, is isomorphic to the localization $QH^*(Gr(m,n))[q^{-1}]$.  Moreover, under this isomorphism, the generators ${\bf \tilde e}^n_r$ and ${\bf \tilde h}^n_r$ map to the Schubert classes $\sigma_{(1^r)}$ and $\sigma_{(r)}$, respectively. In Section \ref{Sec:nTL}, we prove that Postnikov's isomorphism is actually the composition of the parabolic Peterson isomorphism in Theorem \ref{mainintro}, followed by two duality isomorphisms, which should be thought of as corrections addressing the fact that the parabolic Peterson isomorphism itself does \emph{not} actually directly map ${\bf \tilde h}_r \mapsto \sigma_{(r)}$.  The following result combines Theorems \ref{modpet} and \ref{algiso} with Corollary \ref{decomp} in the body of the paper.

\begin{thm}\label{PostIntro}
	Let $P$ be the maximal parabolic subgroup such that $SL_n(\CC)/P \cong Gr(m,n)$.  The subalgebra $\mathbb X$ of the localized affine nil-Templerley-Lieb algebra $\anTLkn$ generated by ${\bf \tilde e}^n_r$ and/or ${\bf \tilde h}^n_r$ is isomorphic to the quotient of the affine Fomin-Stanley algebra $\mathbb B$ appearing in Theorem \ref{mainintro}:
	\begin{equation}
	\iso: ( H_*(\mathcal{G}r_{SL_n})/ J_P)[ ({\bf \overline{s}}_{\lambda(u)}^{(k)})^{-1}] \longrightarrow \mathbb X,
	\end{equation}
	and the isomorphism $\iso$ is induced by the natural projection $\mathbb B \longrightarrow \mathbb X$ killing any element which supports a braid relation.
	
	Let $T$ be the transpose map and $\sd$ the strange duality involution reviewed in Theorem \ref{SD}. The following composition is also an isomorphism and maps the generators as follows:
	\begin{equation}
	\begin{aligned}
	T \circ \sd \circ \Psi_P : ( H_*(\mathcal{G}r_{SL_n})/ J_P)[ ({\bf \overline{s}}_{\lambda(u)}^{(k)})^{-1}] &\longrightarrow QH^*(Gr(n-\kk,n))[q^{-1}]\\
	{\bf \tilde e}_r &\longmapsto \sigma_{(1^r)}\\
	{\bf \tilde h}_r &\longmapsto \sigma_{(r)}.
	\end{aligned}
	\end{equation}
	Therefore, if instead $P'$ is the maximal parabolic subgroup such that $SL_n(\CC)/P' \cong Gr(n-m,n)$, then Postnikov's isomorphism is in fact the composition of these isomorphisms
	\begin{equation}
	T \circ \sd \circ \Psi_{P'} \circ \iso^{-1}: \mathbb X \longrightarrow QH^*(Gr(\kk,n))[q^{-1}],
	\end{equation}
	providing an independent proof of \cite[Proposition 8.5]{P05}, which we restate in Theorem~\ref{PostIso}.
\end{thm}

\subsection{Directions for future work}

We provide several examples of results concerning $QH^*(Gr(m,n))$ which would be interesting to understand through the Peterson lens, and for which we expect that the machinery developed in this paper might be especially useful.  In \cite{P05}, Postnikov argues that the affine nilTemperley-Lieb algebra is a Grassmannian analog of Fomin-Kirillov's quadratic algebra from \cite{FK}, in which the Dunkl elements encode the multiplication in $QH^*(SL_n(\CC)/B)$.  Reverse engineering the analog of the subalgebra of the affine nilTemperley-Lieb algebra considered in Theorem \ref{PostIntro} is thus a likely first step toward extending Fomin and Kirillov's approach to other partial flags. 

Lee has proved  that cylindric skew Schur functions are cylindric Schur positive, and that these expansions contain all Gromov-Witten invariants, using a result of Lam from \cite{LamAJM} that the cylindric skew Schur functions are precisely those affine Stanley symmetric functions which are indexed by permutations which do not support braid relations \cite{LeeCylindric}. It would be interesting to further explore these connections between the cylindric Schur functions of Postnikov and the affine Stanley symmetric functions indexed by 321-avoiding affine permutations, perhaps using the crystal structure developed by Morse and Schilling \cite{MorseSchilling}.

\subsection{Organization of the paper}

Section \ref{Sec:review} provides a self-contained introduction to the Peterson isomorphism, including all necessary background on the relevant quantum and affine Schubert calculus, the combinatorics of the associated Weyl groups, and finite and affine root systems. The reader familiar with this background may prefer to consult the published version which provides a condensed treatment of Section \ref{Sec:review}.

We open in Section \ref{Sec:QSch} with a review of quantum Schubert calculus for the Grassmannian, including indexing the Schubert basis and the Pieri rule for multiplying quantum Schubert classes.  We also define the duality isomorphisms appearing in Theorem \ref{PostIntro}. In Section \ref{Sec:HomAffGr}, we review affine Schubert calculus as modeled by the non-commutative $k$-Schur functions.  Following \cite{LamkSchur}, we discuss how the affine Fomin-Stanley algebra $\mathbb B$ generated by the non-commutative homogeneous symmetric functions ${\bf \tilde h_r}$ is isomorphic to the homology of the affine Grassmannian, and that the $k$-Schur basis maps to the affine Schubert basis in $H_*(\mathcal{G}r_{SL_n})$.  Section \ref{Sec:Peterson} provides a self-contained exposition of the Peterson isomorphism in type $A$ both for the complete flag variety and for any partial flag, including the relevant affine root system vocabulary.  

Our results are both formally stated and proved in Sections \ref{PetIso} and \ref{Sec:nTL}.  In Section \ref{PetIso}, we prove Theorem \ref{mainintro}, which specializes the statement of the parabolic Peterson isomorphism to specifically recover the quantum cohomology of the Grassmannian.  As an application, we prove Theorem \ref{PostIntro} in Section \ref{Sec:nTL}, revealing that Postnikov's approach to the quantum cohomology of the Grassmannian via the affine nilTemperley-Lieb algebra is a direct consequence of Theorem \ref{mainintro}.

\subsection*{Acknowledgments}

This paper is an outgrowth of the senior thesis research project conducted by the first author and directed by the second author at Haverford College.  The authors thank Kaisa Taipale for useful discussions about parabolic Peterson and helpful comments on an earlier version of this manuscript, as well as Thomas Lam for a clarification about the ideal $J_P$. The authors wish to thank the anonymous referee for comments which helped to greatly streamline the exposition in the published version.  This project was conceived while the second author was in residence at the Max-Planck-Institut f\"{u}r Mathematik, and she is grateful to the institute for fostering excellent research conditions. The second author was partially supported by NSF grant DMS--1600982.


\section{Background on Schubert Calculus} \label{Sec:review}

The majority of the material reviewed here is discussed in more detail in the book \cite{kschur}, which presents a comprehensive exposition based on the original papers. We provide references to the corresponding material in \cite{kschur} when possible throughout the section, and we refer the reader to the exposition in \cite{kschur} for the original citations.  For a more streamlined treatment of this background, we invite the reader familiar with the material in Section \ref{Sec:review} to consult the version of this paper published in the {\it Journal of Combinatorics}.


\subsection{Quantum cohomology of the Grassmannian}\label{Sec:QSch}

The \textbf{Grassmannian} of $m$-planes in $\CC^n$ is the set $Gr(m,n)$ of all $m$-dimensional subspaces of $\CC^n$.  This set has the geometric structure of a projective algebraic variety, as well as the topological structure of a CW-complex.  The decomposition of $Gr(m,n)$ into Schubert cells arises from a variation of the Bruhat decomposition on $SL_n(\CC)$.  In this case, the Schubert cells are indexed by minimal length coset representatives in the quotient $S_n/(S_m \times S_{n-m})$ of the symmetric group on $n$ letters, as are their closures.  The cohomology $H^*(Gr(m,n))$ admits a $\ZZ$-basis given by the classes of these Schubert varieties, and the multiplication in this ring carries information enumerating subvarieties satisfying certain incidence conditions.  The \textbf{quantum cohomology} of the Grassmannian is a ring $QH^*(Gr(m,n))$ which admits a $\ZZ[q]$-basis of Schubert classes again indexed by the elements of $S_n/(S_m \times S_{n-m})$.  The variable $q$ serves as a formal parameter which captures the degree of the curves passing through certain prescribed subvarieties.  In particular, setting $q=0$ recovers the classical cohomology $H^*(Gr(m,n))$.

\subsubsection{Schubert classes in the Grassmannian}

There are many equivalent ways to index the Schubert classes in the Grassmannian.  The symmetric group $S_n$ is generated by the {\bf simple transpositions} $s_i$ which interchange $i$ and $(i+1)$, which means that every element $w \in S_n$ can be written as a product $w = s_{i_1}\cdots s_{i_r}$ where $i_j \in \{1, \dots, n-1\}$.  These generators also satisfy the following relations:
\begin{equation}
\begin{aligned}
s_i^2 &=1\\
s_is_j &= s_js_i \ \text{ if } |i-j| >1, \\
s_is_{i+1}s_i &= s_{i+1}s_is_{i+1}. 
\end{aligned}
\end{equation}
When the product $w=s_{i_1}\cdots s_{i_r}$ is written using the minimum possible number of generators, we say this is a {\bf reduced expression} for $w$, and we define the {\bf length} of $w$ to be $\ell(w) = r$.  We frequently abbreviate $w=s_{i_1}\cdots s_{i_r}$ by recording only the subscripts $w = s_{i_1i_2\cdots i_r}$.

In the quotient $S_n/(S_m \times S_{n-m})$, generators of the subgroup $S_m \times S_{n-m}$ appearing on the right are absorbed, which means that minimal length coset representatives have the property that all reduced expressions must end in $s_m$.  As permutations in {\bf one-line notation} recording only the images of the elements $[n] = \{1, 2, \dots, n\}$ under the bijection on $[n]$, the minimal length coset representatives of $S_n/(S_m \times S_{n-m})$ will have the form $[a_1\ a_2\ \cdots \ a_m\ \vert\ a_{m+1} \cdots \ a_n]$, where $a_1 < a_2 < \cdots a_m$ and $a_{m+1} < \cdots < a_n$, with a single (possible) descent at $a_{m+1}> a_m$ in the $m^{\text{th}}$ position.  To each such permutation in one-line notation, we can associate a {\bf binary string}, or a {\bf $01$-word}, by recording whether each element of $[n]$ is in the right or left batch in the one-line notation.  More concretely, going through the elements of $[n]$ in order, record a 0 if the number is to the right of the vertical line at position $m$, and record a 1 if the number is to the left. Equivalently, this $01$-word traces out a {\bf Young diagram} which fits inside an $m \times (n-m)$ rectangle by starting in the bottom left corner and taking a right step for each 0 and an upward step for each 1, reading the $01$-word left to right.  A Young diagram is typically recorded as a {\bf partition}, or a weakly decreasing sequence of positive integers $\lambda \in \NN^m$ which lists the number of left-justified boxes in each row of the diagram recorded from top to bottom.

\begin{ex} Let $m=5$ and $n=9$.  Here is an example which illustrates the correspondences between the reduced expression, the one-line notation, the binary string recorded by the $01$-word, the Young diagram, and the partition for a permutation in $S_9/(S_5 \times S_4)$:
\begin{equation*}
\text{\Large $s_{124354765}$}  \ \ \longleftrightarrow \ \  [\textcolor{red}{2\ 3\ 5\ 6\ 8}\ \vert\ \textcolor{blue}{1\ 4\ 7\ 9} ] \ \  \longleftrightarrow \ \  \textcolor{blue}{0}\textcolor{red}{11}\textcolor{blue}{0}\textcolor{red}{11}\textcolor{blue}{0}\textcolor{red}{1}\textcolor{blue}{0} \ \  \longleftrightarrow \ \  {\scriptsize \tableau[Ys]{ & & \\ &  \\ & \\ \\   \ }}
\ \  \longleftrightarrow \quad (3,2,2,1,1). 
\end{equation*}
\end{ex}

The most common way to represent an arbitrary Schubert class in $QH^*Gr(m,n)$ is certainly as a Young diagram fitting inside an $m\times (n-m)$ rectangle or equivalently by the corresponding partition.  We thus denote by $\mathcal{P}_{mn}$ the set of all such diagrams.  Given $\lambda \in \mathcal P_{mn}$, when necessary we denote the parts of the corresponding partition by $\lambda = (\lambda_1, \dots, \lambda_m) \in \NN^m$, where the last few parts are permitted to equal zero.

\subsubsection{Quantum Schubert calculus}

In quantum Schubert calculus, the primary goal is to understand the product structure in the ring $QH^*(Gr(m,n))$.  Namely, given any two partitions $\lambda, \mu \in \mathcal P_{mn}$, we aim to expand the product
\begin{equation}
\sigma_\lambda * \sigma_\mu = \sum\limits_{\substack{\nu \in \mathcal P_{mn},\\ d \in \ZZ_{\geq 0}}} c_{\lambda, \mu}^{\nu, d}q^d\sigma_\nu
\end{equation}
in terms of the Schubert basis.  The {\bf quantum Littlewood-Richardson coefficients} $c_{\lambda, \mu}^{\nu, d} \in \ZZ_{\geq 0}$ are three-point genus zero Gromov-Witten invariants and are thus known to be nonnegative integers due to the questions in enumerative geometry which they answer.  A complete combinatorial quantum Littlewood-Richardson rule was given by Buch, Kresch, Purbhoo, and Tamvakis in \cite{BKPT} by counting the number of puzzles corresponding to a certain 2-step flag variety.  

For easy comparison with the examples discussed in this paper, we briefly review the quantum Pieri formula of \cite{Bertram}, which yields a recursive method for calculating any quantum product in $QH^*(Gr(m,n))$.  For the reformulation restated below, as well as for various discussions throughout the paper, we require the combinatorial language of hook shapes.  Any box in a Young digram has an associated \textbf{hook}, consisting of all boxes to the right and below the given box, including the box itself; see the diagram on the left in Figure \ref{fig:hooks}.  
\begin{figure}[h]
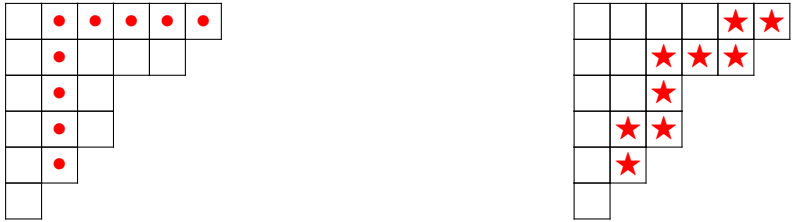

\begin{minipage}[c]{0.45\textwidth}
\[ \tableau[Ys]{ & \textcolor{red}{\bullet} & \textcolor{red}{\bullet} & \textcolor{red}{\bullet} &\textcolor{red}{\bullet} & \textcolor{red}{\bullet} \\ & \textcolor{red}{\bullet} & & &   \\  & \textcolor{red}{\bullet} &   \\ &\textcolor{red}{\bullet}  &  \\ &\textcolor{red}{\bullet}  \\   \\ } 
\]
\end{minipage}
\begin{minipage}[c]{0.45\textwidth}
\[ \tableau[Ys]{& & & &\textcolor{red}{\bigstar} & \textcolor{red}{\bigstar} \\ & & \textcolor{red}{\bigstar}  & \textcolor{red}{\bigstar}  & \textcolor{red}{\bigstar}  \\  & &\textcolor{red}{\bigstar}  \\ &\textcolor{red}{\bigstar}  &\textcolor{red}{\bigstar}  \\ &\textcolor{red}{\bigstar}  \\  \\  } 
\]
\end{minipage}
\caption{A 9-hook and the corresponding 9-rim hook.}\label{fig:hooks}
\end{figure}
The {\bf hook length} equals the number of boxes in the hook, and if there are $n$ such boxes we refer to the shape as an \textbf{$n$-hook}.  Each $n$-hook also corresponds to an \textbf{$n$-rim hook}, consisting of the $n$ contiguous boxes running along the border of the diagram which connect the top rightmost and bottom leftmost boxes of the $n$-hook; see the right diagram in Figure \ref{fig:hooks}.

In the case of the Grassmannian, knowing how to multiply each of the Schubert classes indexed by a single row of boxes $(r)$ for all $1 \leq r \leq n-m$ completely determines the multiplication in $QH^*(Gr(m,n))$; similarly for the column shapes.  By combining the classical Pieri rule for multiplying by classes of row shapes with the rim hook rule of \cite{BCFF}, one obtains the following restatement of the quantum Pieri rule originally due to Bertram.

\begin{thm}[Quantum Pieri Rule \cite{Bertram}]\label{qPieri}
Let $\lambda \in \mathcal P_{mn}$ and fix an integer $1 \leq r \leq n-m$.  In $QH^*(Gr(m,n))$,
\begin{equation}
\sigma_{(r)} * \sigma_\lambda = \sum\sigma_\mu + q \sum \sigma_\nu,
\end{equation}
where the first sum is over all partitions $\mu\in \mathcal P_{mn}$ such that $\mu$ is obtained from $\lambda$ by adding exactly $r$ boxes, no two of which were added in the same column; and the second sum is over all partitions $\nu \in\mathcal P_{mn}$ such that $\nu$ is obtained from $\lambda$ by adding exactly $r$ boxes, none of which are added to the right of column $n-m$ and no two in the same column, and then removing a single $n$-rim hook.
\end{thm}

\subsubsection{Duality isomorphisms in quantum cohomology}

In order to make a precise connection later between the affine Fomin-Stanley algebra and the affine nilTemperley-Lieb algebra, we will  require several isomorphisms on the quantum cohomology of the Grassmannian arising from various types of duality, which we now review.

First, the cohomology of the Grassmannian admits a natural duality isomorphism arising from the fact that $Gr(m,n) \cong Gr(n-m,n)$.  Given a partition $\lambda \in \mathcal P_{mn}$, define the {\bf transpose} of $\lambda$ by interchanging the role of the rows and columns, and denote it by $\lambda^T \in \mathcal{P}_{n-m,n}.$
 The map 
\begin{equation}\label{transpose}
\begin{aligned}
T: QH^*(Gr(m,n)) & \longrightarrow QH^*(Gr(n-m,n)) \\
 \sigma_\lambda & \longmapsto \sigma_{\lambda^T}
 \end{aligned}
\end{equation}
is an isomorphism.  Note in particular that the transpose interchanges the classes $\sigma_{(r)}$ and $\sigma_{(1^r)}$.

Second, the Schubert basis is self-dual with respect to the Poincar\'{e} pairing, which induces an automorphism on cohomology.  More specifically, if $\lambda = (\lambda_1, \dots, \lambda_m) \in \mathcal P_{mn}$, denote by 
\begin{equation}
\lambda^\vee = (n-m-\lambda_m, \dots, n-m-\lambda_1)
\end{equation}
 the {\bf complement} of $\lambda$, which is the shape given by looking at the complement of $\lambda$ in the full $m\times (n-m)$ rectangle, and then rotating the shape 180 degrees.  The following map is an automorphism
\begin{equation}\label{transpose}
\begin{aligned}
QH^*(Gr(m,n)) & \longrightarrow QH^*(Gr(m,n)) \\
 \sigma_\lambda & \longmapsto \sigma_{\lambda^\vee}.
 \end{aligned}
\end{equation}

Finally, we review the {\bf strange duality} involution appearing in Postnikov \cite{P05}, which was discovered independently in the special case of $q=1$ by Hengelbrock \cite{Hengel}, and later generalized by Chaput, Manivel, and Perrin to other (co)miniscule homogeneous varieties \cite{CMP07}.

\begin{thm}[Theorem 6.5 \cite{P05}] There is an involution
\begin{equation}
\begin{aligned}
\sd: QH^*(Gr(\kk,n)) & \longrightarrow QH^*(Gr(\kk,n)) \quad \text{by} \\
\sigma_\lambda & \longmapsto q^{-\text{diag}_\kk(\lambda)}\sigma_{\mu} \\
 q\ &\longmapsto q^{-1}
 \end{aligned}
\end{equation}
where if the $01$-word corresponding to $\lambda$ is $\operatorname{b} = i_1i_2\cdots i_n$, then the $01$-word corresponding to $\mu$ is $\bin' = i_\kk i_{\kk-1}\cdots i_1 i_n i_{n-1}\cdots i_{\kk+1}$, and $\text{diag}_\kk(\lambda)$ is the number of boxes on the main diagonal. \label{SD}
\end{thm}

In the involution $\sd$, the map on the shapes $\lambda \mapsto \mu$ is given by taking the complement of the southwest and northeast sections of $\lambda$ separately. Here, we define the southwest section of $\lambda$ to be the rectangle below the largest {\bf Durfee square} (meaning upper-left justified)  that fits in the Young diagram, and the northeast section is the rectangle to the right of the largest Durfee square. Note that  $\text{diag}_\kk(\lambda)$ is the size of this Durfee square by definition.  We illustrate this comparatively strange involution with an example below.

\begin{ex} Consider $\text{Gr}(5,9)$ and the partition $\lambda = (3,2,2,1,1)$.
The 01-word for $\lambda$ is given by $\bin=\textcolor{red}{01101}\textcolor{blue}{1010}$, where we have separated the first $m=5$ and last $n-m=4$ bits by the colors red and blue, respectively.  Therefore, the image $\mu$ has the 01-word $\textcolor{red}{10110}\textcolor{blue}{0101}$ given by reversing the order of each of the red and blue strings. The corresponding shape is $\mu = (4,3,1,1,0)$, so we can compare the two Young diagrams:
\begin{equation}\lambda = {\scriptsize
\tableau[Ys]{ \times & \times & & \fl\\ \times &\times & \fl &\fl \\ & \\ & \fl \\ & \fl \ }} \qquad\qquad \mu = {\scriptsize \tableau[Ys]{ \times&\times & &\\\times & \times& &\fl \\ &\fl \\ & \fl \\ \fl & \fl	}}
\end{equation}
We have marked by $\times$ the boxes corresponding to the largest Durfee square in each of the shapes.  We have also added black squares to fill the southwest  $3 \times 2$ and northeast $2 \times 2$ rectangles below and to the right of the $2 \times 2$ Durfee square. We can then see that the shape in $\mu$'s southwest section is the complement in the $3 \times 2$ rectangle of the shape in $\lambda$'s southwest section, and similarly for the northwest section. Through the involution, we would thus have $\sigma_\lambda \mapsto q^{-2} \sigma_\mu$.
\end{ex}


\subsection{Homology of the affine Grassmannian}\label{Sec:HomAffGr}

Let $F = \CC((t))$ be the field of Laurent series with complex coefficients, and let $\mathcal{O} = \CC[[t]]$ be the ring of integers in $F$.  The \textbf{type $A$ affine Grassmannian} is defined to be \[\mathcal{G}r = SL_n(F)/SL_n(\mathcal{O}),\] and this space admits a decomposition into affine Schubert cells which are indexed by minimal length coset representatives in the quotient of the affine symmetric group $\tilde S_n$ by the symmetric group $S_n$.  As in the non-affine case, the (co)homology classes of these affine Schubert cells form a $\ZZ$-basis for the (co)homology of the affine Grassmannian.  The Peterson isomorphism says that, up to localization, the affine Schubert classes in the homology of the affine Grassmannian map to the quantum Schubert classes for the complete complex flag variety.  We review the precise statement in Section \ref{Sec:Peterson}, but before this isomorphism can be useful for doing calculations in quantum cohomology, we need to review the algebra and combinatorics of affine Schubert calculus.  

Bott showed that the homology of the type $A$ affine Grassmannian $H_*(\mathcal{G}r)$ is isomorphic to a subring of the ring of symmetric functions \cite{Bott}, and Lam proved in \cite{LamkSchur} that the $k$-Schur functions of Lapointe and Morse \cite{LMkSchur} represent the Schubert classes in $H_*(\mathcal{G}r)$ under Bott's isomorphism.  Lam's approach identifies both the homology of the affine Grassmannian and the subalgebra of symmetric functions spanned by the $k$-Schur functions with the affine Fomin-Stanley subalgebra of Kostant and Kumar's nilHecke ring \cite{KK}.  Following this approach, in this section we define a basis of the affine Fomin-Stanley subalgebra which can be identified through Lam's isomorphism with the Schubert basis for $H_*(\mathcal{G}r)$.

\subsubsection{Affine Schubert classes}\label{Sec:AffSchClasses}

For the remainder of this paper, we fix an integer $n \in \NN$ and define $k = n-1$.  In this section, we develop the algebra and combinatorics of the indexing sets for the affine Schubert cells in $\mathcal{G}r$.

The {\bf affine symmetric group}, denoted $\tilde S_n$, is obtained from $S_n$ by adding a generator
$s_0$ that satisfies the following relations with the simple transpositions of $S_n$
\begin{equation}
\begin{aligned}
s_0^2 &= 1\\
s_0 s_i & = s_i s_0 \text{ \ \ if $i \not \in \{1,n-1\}$}\\
s_0 s_{\pm 1}s_0 &= s_{\pm 1} s_0 s_{\pm 1}
\end{aligned}
\end{equation}
where the indices are taken modulo $n$. The definitions of reduced expressions for elements $w \in  \tilde S_n$ and the length $\ell(w)$ naturally extend to $\tilde S_n$.  We can consider the quotient 
$\tilde S_n / S_n$, and our standard coset representatives will be those of minimal length; namely, the elements $w \in \tilde S_n$ whose reduced expressions $w = s_{i_1,i_2,...,i_r}$ are forced to end in $s_0 = s_{i_r}$. Denote by $\tilde S_n^0$ this set of minimal length coset representatives in $\tilde S_n/S_n$, and note that we will sometimes abuse notation and write $\tilde S^0_n \subset \tilde S_n$.

There are several alternative indexing sets for the affine Schubert cells involving partitions satisfying certain properties.  The {\bf $k$-bounded partitions} are all Young diagrams with width less than $k=n-1$.  The {\bf n-cores} are all Young diagrams with no removable $n$-rim hook. 

\begin{thm}[Section 1.2 \cite{kschur}]\label{corekbounded}
The elements of $\tilde S_n^0$ are in bijection with both the set of $k$-bounded partitions, as well as the set of $n$-cores.
\end{thm}

We briefly review the bijections used in the proof of Theorem \ref{corekbounded}. To find the corresponding $n$-core from a $k$-bounded partition, remove any box with hook length greater than $n$. Then slide the partition to be left-justified. To reverse the procedure, start at the bottom row of the Young diagram. For each row, slide the row (and those above) to the right until no box in the row has a hook length that is greater than $k$.

To find the $n$-core $\mu$ corresponding to an element $w=s_{i_1}s_{i_2}\cdots s_{i_m}s_0 \in  \tilde S_n^0$, first draw a grid of boxes as below:
\begin{equation}
\begin{array}{cccc}
0 & 1 & 2 & \cdots \\
n-1 & 0 & 1 & \cdots \\
n-2 & n-1 & 0 & \ddots \\
\vdots & \vdots & \ddots & \ddots
\end{array}\label{grid}
\end{equation}
Reading right to left, place a box around the indices $0$, $i_m$, $i_{m-1}$, etc. At each step, place \emph{all} possible boxes with the same label which keep the shape a Young diagram. To go from an $n$-core $\mu$ to an element of $\tilde S_n^0$, overlay this grid on the diagram $\mu$. At the $j^{\text{th}}$ step, find an index $i_j$ such that removing all boxes on the border of the Young diagram $\mu^j$ labeled by $i_j$ gives a new Young diagram $\mu^{j+1}$, where $\mu^0 = \mu$. Once all boxes have been removed at the $r^{\text{th}}$ step, the element $w=s_{i_1,i_2,...,i_{r-1},0} \in \tilde S_n^0$ is the reduced expression corresponding to $\mu$. Note that $w$ necessarily ends in $s_0$ by construction, since 0 is the label of the box in the upper leftmost corner of $\mu$.

\begin{ex}
Let $n=4$ so that $k=3$, and let $\lambda=(2,1,1)$ be a $3$-bounded partition. Since the top left box has a hook length of four, we must slide the top row to the right until there are no hook lengths of four or more to obtain the corresponding $4$-core $\mu = (3,1,1)$. Now, removing boxes:
\begin{equation*}
{\small \tableau[Ys]{0  & 1 & 2  & \cdots \bl \\
		3 & 0 \bl & 1 \bl & \cdots \bl \\
		2 & 3 \bl & 0 \bl & \ddots \bl \\
		\vdots \bl & \vdots \bl & \ddots \bl & \ddots \bl}
}
\xrightarrow{\ s_2\ }
{\small \tableau[Ys]{0  & 1 & 2 \bl & \cdots \bl \\
		3 & 0 \bl & 1 \bl & \cdots \bl \\
		2 \bl & 3 \bl & 0 \bl & \ddots \bl \\
		\vdots \bl & \vdots \bl & \ddots \bl & \ddots \bl}
}
\xrightarrow{s_2s_1}
{\small \tableau[Ys]{0  & 1 \bl & 2 \bl & \cdots \bl \\
		3 & 0 \bl & 1 \bl & \cdots \bl \\
		2 \bl & 3 \bl & 0 \bl & \ddots \bl \\
		\vdots \bl & \vdots \bl & \ddots \bl & \ddots \bl}
}
\xrightarrow{s_2s_1 s_3}
{\small \tableau[Ys]{0  & 1 \bl & 2 \bl & \cdots \bl \\
		3 \bl & 0 \bl & 1 \bl & \cdots \bl \\
		2 \bl & 3 \bl & 0 \bl & \ddots \bl \\
		\vdots \bl & \vdots \bl & \ddots \bl & \ddots \bl}
}
\xrightarrow{s_2s_1s_3s_0}
{\small \tableau[Ys]{0 \bl & 1 \bl & 2 \bl & \cdots \bl \\
		3 \bl & 0 \bl & 1 \bl & \cdots \bl \\
		2 \bl & 3 \bl & 0 \bl & \ddots \bl \\
		\vdots \bl & \vdots \bl & \ddots \bl & \ddots \bl}
}
\end{equation*}
The $4$-core $\mu= (3,1,1)$ thus corresponds to $s_2 s_1 s_3 s_0 \in \tilde S_n^0$.  To go from $s_2 s_1 s_3 s_0$ back to $\mu$, read the arrows backwards to add the boxes corresponding to each letter in $s_2 s_1 s_3 s_0$ from right to left. 
\end{ex}

Finally, the transpose map on $n$-cores induces a bijection on $k$-bounded partitions. Letting ${\bf c}(\lambda)$ denote the $n$-core corresponding to a $k$-bounded partition $\lambda$, define the  {\bf $k$-conjugate} of 
$\lambda$ to be the $k$-bounded partition $\lambda^{\omega_k} = {\bf c}^{-1}[({\bf c(\lambda)})^T]$. Note that $k$-conjugation maps the reduced expression $w =s_{i_1,i_2,...,i_r,0}\in \tilde S_n^0$ to $w'=s_{i_1',i_2',...,i_r',0}$
where $i_j' = n-i_j$ for all $1 \leq j\leq r$.

\begin{ex} We compute the $4$-conjugate of $\lambda = (4,3,2,2)$. We label the boxes with their hook lengths for clarity. 
	\begin{equation*}\scriptsize
	\tableau[Ys]{7& 6& 3& 1\\ 5& 4& 1\\ 3& 2 \\ 2& 1 } \quad {\LARGE \xrightarrow{\ {\bf c}\ }} \quad 
	\tableau[Ys]{\fl &\fl & \fl & \fl & \fl &4& 3& 2& 1 \\ \fl & \fl & 3& 2& 1\\ 3& 2 \\ 2& 1\ } \quad{\LARGE \xrightarrow{\ T\ }} \quad
	\tableau[Ys]{\textcolor{red}{12} &\textcolor{red}{7} &3 &2 \\\textcolor{red}{11} &\textcolor{red}{6} &2 &1 \\\textcolor{red}{8} & 3 \\\textcolor{red}{7} & 2\\\textcolor{red}{6} & 1\\ 4\\ 3\\ 2\\ 1 } \quad {\LARGE \xrightarrow{\ {\bf c}^{-1} }} \quad
	\tableau[Ys]{10& 2\\9 & 1 \\7 \\6 \\5 \\ 4\\ 3\\ 2\\ 1}
	\normalsize
	\end{equation*}
	The black squares show the ones that appear from sliding to obtain the 5-core, and the red numbers get removed in the passage back to the 4-bounded partition after the transpose. 
\end{ex}

\subsubsection{$k$-Schur functions}

The ring of symmetric functions $\Lambda \subset \ZZ[x_1,x_2,...]$ is generated by the
{\bf homogeneous symmetric functions} which are defined by
\begin{equation}
h_r = \sum_{\footnotesize 0\le i_1 \le i_2 \le \cdots \le i_r} x_{i_1}x_{i_2}\cdots x_{i_r}.
\end{equation}
The ring $\Lambda$ also has a basis of {\bf Schur functions} denoted $s_\lambda$, which are indexed by a partition
$\lambda$ and follow a Pieri rule like the Schubert classes $\sigma_{\lambda} \in H^*(Gr(\kk,n))$.
We can consider instead the set $\Lambda^{(k)} = \ZZ[h_1,...,h_k]$, which has a basis given by the {\bf $k$-Schur functions} denoted $s_{\lambda}^{(k)}$.  As noted in \cite{LM07}, in the limit $k\to \infty$, the $k$-Schur functions become the Schur functions. We will define the $k$-Schur functions recursively via the weak Pieri rule below.

\begin{thm}[Weak Pieri rule, Section 2.2 \cite{kschur}] \label{WeakPieri}
	Let $\lambda$ be a $k$-bounded partition, and fix an integer $1 \le r \le k$. Then
	\begin{equation}
	h_r \cdot s_{\lambda}^{(k)} = \sum s_{\nu}^{(k)}
	\end{equation}
	where the sum is over all $k$-bounded partitions $\nu$ satisfying two conditions: $\nu$ is obtained from $\lambda$ by adding exactly $r$ boxes, no two of which were added in the same column, and $\nu^{\omega_k}$ is obtained from $\lambda^{\omega_k}$ by adding exactly $r$ boxes, no two of which were added to the same row.
\end{thm}

To begin the recursion, first note that  $s_{(r)}^{(k)} =h_r$ for all $r\le k$. To compute other $k$-Schur functions, we use the following lemma, the inductive proof for which is a straightforward generalization of Example 2.11 in \cite{kschur}.

\begin{lemma}[Section 2.2 \cite{kschur}]\label{hdecomp} 
Let $\lambda = (\lambda_1,...,\lambda_p)$ be a $k$-bounded partition. Then,
	\begin{equation}\label{hdecomp2}
	s_\lambda^{(k)}= h_{\lambda_1} \cdot s_{(\lambda_2,...,\lambda_p)}^{(k)}+ \sum_{i = 1}^{n-1-\lambda_1} h_{i+\lambda_1} \cdot\left(\sum_{j} c_{ij} s_{\mu_{ij}}^{(k)}\right)
	\end{equation}
	for some finite number of k-bounded partitions $\mu_{ij}$ and constants $c_{ij} \in \ZZ$. 
\end{lemma}

The above lemma gives a definition of the $k$-Schur functions as some polynomial
in the $h_r$, which can be determined by computing specific products and then rearranging the
expressions. If we want to compute $s_\lambda^{(k)}$, we first compute the product
$h_{\lambda_1} s_{(\lambda_2,...,\lambda_p)}^{(k)}$. The $k$-Schur function $s_\lambda^{(k)}$ will appear in this product, as well as other $k$-Schur functions, which we proceed to compute recursively in the same manner.  We illustrate this recursion via example.

\begin{ex}
Let $n=5$ so that $k=4$, and let $\lambda = (4,3,1,1)$. We will compute $s_{(4,3,1,1)}^{(4)}$ in terms of the $h_r$.  Lemma \ref{hdecomp} says that we should first compute the product 
	$h_4 \cdot s_{(3,1,1)}^{(4)} = s_{(4,3,1,1)}^{(4)}$, which we do using the weak Pieri rule of Theorem \ref{WeakPieri}.  Using the weak Pieri rule again,  the lemma says that we should next compute 
	$h_3 \cdot s_{(1,1)}^{(4)} = s_{(3,1,1)}^{(4)}$.  To calculate this product, note that $(1^5)=(4,1)^{\omega_4}$ cannot be obtained from $(2)=(1,1)^{\omega_4}$ by adding three boxes 
	to different rows. Last, $h_1 \cdot s_{(1)}^{(4)} = s_{(1,1)}^{(4)} + s_{(2)}^{(4)}$, since  both $(1)^{\omega_4}=(1)$ and $(1,1)^{\omega_4}=(2)$. Recalling that $s_{(r)}^{(4)}=h_r$ for all $r \leq 4$,
	we rearrange the products to find 
\begin{equation}\label{kSchurExEq}
	s_{(1,1)}^{(4)}=h_1^2 - h_2 \qquad\quad  s_{(3,1,1)}^{(4)}=h_3(h_1^2-h_2) \qquad \quad s_{(4,3,1,1)}^{(4)} = h_4h_3(h_1^2-h_2).
\end{equation}
	\label{commkschurex}
\end{ex}

Similar to the Schubert classes $\sigma_\lambda$, which yield an automorphism of $ QH^*(Gr(\kk,n))$ given by $\sigma_{\lambda} \mapsto \sigma_{\lambda^T}$, the $k$-Schur functions also admit an automorphism under $k$-conjugation.

\begin{thm}[Theorem 38 \cite{LM07}]\label{kconjinv}
The $k$-conjugation map is an involution on the subalgebra of symmetric functions spanned by the homogeneous symmetric functions: 
	\begin{equation}
	\begin{aligned}
	\omega : \Lambda^{(k)} &\longrightarrow \Lambda^{(k)}\\
	s_{\lambda}^{(k)} &\longmapsto s_{\lambda^{\omega_k}}^{(k)}.
	\end{aligned} 
	\end{equation}
\end{thm}

\subsubsection{Affine nilCoxeter algebra and non-commutative $k$-Schur functions}

The homology of the affine Grassmannian is isomorphic to a subalgebra of the affine nilCoxeter algebra, in which the affine Schubert classes correspond to the non-commutative $k$-Schur functions, and so we now define the affine nilCoxeter algebra and identify this distinguished basis.

\begin{Def} The {\bf affine nilCoxeter algebra} $(\mathbb A_\text{af})_0$ is generated by $1$ and elements $A_i$ for $i\in \{0,1,...,n-1\}$ satisfying the following relations:
\begin{equation}
\begin{aligned}
A_i^2&=0  \\
A_i A_j &= A_j A_i	&& \iff (i - j)\operatorname{mod} n \not \in \{-1, 1\} \\
A_i A_{i+1} A_i &= A_{i+1} A_i A_{i+1}&&\text{ for all $i$ with indices modulo n}.
\end{aligned}
\end{equation}
More generally, since the simple generators $s_i \in S_n$ satisfy the same relations, we write $A_{w} = A_{w_1}A_{w_2}\cdots A_{w_p}$ where $w= s_{w_1}\cdots s_{w_p}$ or $w=(w_1,...,w_p)$,  is a reduced expression or \textbf{word}, respectively, for $w \in S_n$.
\end{Def}

Note that if we also include generators indexed by the weight lattice together with the corresponding relations, one obtains the nilHecke ring of Kostant and Kumar \cite{KK}. 

\begin{Def} Let $A$ be an algebra with generators $u_0, u_1,..., u_{n-1}$ that satisfy 
\begin{equation}
u_{i}u_{j} = u_{j}u_i \iff (i - j)\operatorname{mod} n \not \in \{-1, 1\},
\end{equation}
with indices taken modulo $n$. Let  $I=\{i_1,...,i_r\} \subsetneq \ZZ/n\ZZ$ be any proper subset.  We define the {\bf cyclically increasing} (resp.\ {\bf cyclically decreasing}) element of $A$ corresponding to $I$, denoted $u_I$ (resp.\ $u_I'$), to be the unique element satisfying 
$u_I = u_{i_1}u_{i_2} \cdots u_{i_r}$ (resp.\ $u_I'= u_{i_1} u_{i_2} \cdots u_{i_r}$) if and only if $i_a = i_b + 1$ implies $a >b$ (resp.\ $b > a$).
\end{Def}

This notion of cyclically increasing and decreasing elements in $(\mathbb A_\text{af})_0$ allows us to generalize the notion of the homogeneous and elementary symmetric functions, which form the basis for an important subalgebra of the affine nilCoxeter algebra.

\begin{Def}[Chapter 3, Section 8.2 \cite{kschur}]\label{hDef} For any $r \in \{ 1,...,n-1\}$, define the {\bf non-commutative homogeneous} and {\bf elementary symmetric functions} in $(\mathbb A_\text{af})_0$ by
\begin{equation} \label{kschurh}
{\bf \tilde h}_r = \sum_{\scriptsize\begin{array}{c}\text{ $w \in \tilde S_n$, $\len(w) = r$} \\ \text{$A_w$ is cyclically decreasing} \end{array}} A_w
\qquad \text{and} \qquad
{\bf \tilde e}_r = \sum_{\scriptsize\begin{array}{c}\text{ $w \in \tilde S_n$, $\len(w) = r$} \\ \text{$A_w$ is cyclically increasing} \end{array}} A_w.
\end{equation}
\end{Def}

\begin{ex} For $n = 4$, we have
	\begin{equation*}
	\begin{aligned}
	{\bf \tilde e}_1 = {\bf \tilde h}_1 &= A_0 + A_1 + A_2 + A_3 \\
	{\bf \tilde h}_2 & = A_{10}+A_{21}+A_{32} +A_{03} +A_{20} +A_{31}\\
	{\bf \tilde h}_3 &= A_{210}+A_{321}+A_{032} + A_{103}.
	\end{aligned}
	\end{equation*}
\label{haffnh}
\end{ex}

We now use the $k$-Schur functions discussed in the previous section to define the non-commutative $k$-Schur functions by replacing each of the homogeneous symmetric functions $h_i$ with their non-commutative counterpart ${\bf \tilde h}_i$.

\begin{Def}[Chapter 3, Section 8.2 \cite{kschur}] \label{noncomkschur}
Let $s_\lambda^{(k)} = f( h_1,...,h_k) $ be a $k$-Schur function with $f(x_1,..,x_k) \in \ZZ[x_1,...,x_k]$. The {\bf non-commutative $k$-Schur function} ${\bf s}_\lambda^{(k)} \in (\mathbb A_\text{af})_0$ is defined to be ${\bf s}_\lambda^{(k)} = f({\bf \tilde h}_1,...,{\bf \tilde h}_k)$. That is, simply express $s_\lambda^{(k)}$ in terms of $h_i$, and then replace $h_i$ by ${\bf \tilde h_i}$.
\end{Def}

\begin{ex} Recall from Eq.~\eqref{kSchurExEq} in Example \ref{commkschurex} that $s_{(1,1)}^{(4)} =h_1^2-h_2$, $s_{(3,1,1)}^{(4)} = h_3(h_1^2-h_2)$, and $s_{(4,3,1,1)}^{(4)}=h_4h_3(h_1^2-h_2)$. Then,
	${\bf s}_{(1,1)}^{(4)} ={\bf \tilde h}_1^2-{\bf \tilde h}_2$, ${\bf s}_{(3,1,1)}^{(4)} = {\bf \tilde h}_3({\bf \tilde h}_1^2-{\bf \tilde h}_2)$, and ${\bf s}_{(4,3,1,1)}^{(4)}={\bf \tilde h}_4{\bf \tilde h}_3({\bf \tilde h}_1^2-{\bf \tilde h}_2)$. Substituting in ${\bf \tilde h}_i$, we can write these non-commutative $k$-Schurs in terms of the generators $A_i$. For example, using Definition \ref{hDef}, we have
	\begin{equation*}
	{\bf s}_{(1,1)}^{(4)} = A_{01} + A_{12} + A_{23} + A_{34} + A_{40} +A_{02}  + A_{03} 
	+ A_{13} + A_{14} + A_{24}.
	\end{equation*}
\label{noncommkschur}
\end{ex}

\begin{prop}[Proposition 8.8, Theorem 8.9 \cite{kschur}]\label{AffFominStanley}
The subalgebra $\mathbb{B} \subset (\mathbb A_\text{af})_0$ generated by ${\bf \tilde h_1, \dots, \tilde h_{n-1}}$ is isomorphic to the subalgebra $\Lambda^{(k)}$ of symmetric functions.  Further, the non-commutative $k$-Schur functions ${\bf s}_\lambda^{(k)}$ form a basis of this subalgebra.
\end{prop}

The subalgebra $\mathbb{B}$ defined in Proposition \ref{AffFominStanley} is called the {\bf affine Fomin-Stanley algebra}.

\subsubsection{Affine Fomin-Stanley algebra and the $j$-basis}

The affine Fomin-Stanley algebra has another basis, which is often referred to as the {\bf $j$-basis}.  
\begin{thm}[Theorem 8.2 \cite{kschur}]\label{T:jBasisExpansion}
There is a basis $\{ j_w^0 \mid w \in \tilde{S}^0_n\}$ for $\mathbb{B}$ which is uniquely defined by the following property
\begin{equation}\label{jbaseform}
j_w^0 = A_w + \sum_{\substack{v \not \in \tilde S_n^0 \\  \len(w) = \len(v)}} c^v_{w} A_v,
\end{equation}
for some constants $c^v_{w} \in \ZZ$, and $j_w^0$ is the unique $j$-basis element containing the term $A_w$. 
\end{thm}
 The $j$-basis is the one which Peterson identifies with the Schubert basis for $H_*(\mathcal{G}r)$, and Lam then proves in \cite{LamkSchur} that the $j$-basis and the basis of non-commutative $k$-Schurs coincide.

\begin{thm}[Theorems 7.1 and 7.4 \cite{LamkSchur}]\label{jBasiskSchur}
For $w \in \tilde{S}_n^0$, let $\lambda(w)$ denote the $k$-bounded partition corresponding to $w$.  As elements of the affine Fomin-Stanley algebra $\mathbb{B}$,
\[ j^0_w = {\bf s}_{\lambda(w)}^{(k)}.\]  Moreover, both bases represent the Schubert classes under the isomorphism $H_*(\mathcal{G}r) \cong \mathbb{B}$.
\end{thm}

\begin{ex} Rewriting the non-commutative $4$-Schurs from Example \ref{noncommkschur}, we have that
	\begin{equation*}
	\begin{aligned}
	j_{s_0}^0 & = {\bf s}_{(1)}^{(4)} = {\bf \tilde h}_1= A_0 + A_1+A_2+A_3+A_4 \\ 
	 j_{s_{40}}^0 & = {\bf s}_{(1,1)}^{(4)} = {\bf \tilde h}^2_1-{\bf \tilde h}_2 = A_{01} + A_{12} + A_{23} + A_{34} + A_{40} +A_{02}  + A_{03} 
	+ A_{13} + A_{14} + A_{24}\\ 
	j_{s_{34210}}^0 & = {\bf s}_{(3,1,1)}^{(4)} ={\bf \tilde h}_3 ( {\bf \tilde h}^2_1-{\bf \tilde h}_2)
	\qquad \text{and} \qquad
	j^0_{s_{210434210}}={\bf s}_{(4,3,1,1)}^{(4)} = {\bf \tilde h}_4 {\bf \tilde h}_3 ( {\bf \tilde h}^2_1-{\bf \tilde h}_2).
	\end{aligned}
	\end{equation*}
\end{ex}

Theorem \ref{jBasiskSchur} says that to perform calculations in the homology of the affine Grassmannian, it suffices to be able to compute products of the non-commutative $k$-Schur functions.  We thus adopt this purely algebraic and combinatorial approach to studying $H_*(\mathcal{G}r)$ throughout this paper.


\subsection{The Peterson isomorphism}\label{Sec:Peterson}

The Peterson isomorphism is an isomorphism between the (equivariant) homology of the affine Grassmannian of a simply-connected complex algebraic group $G$ and the (equivariant) quantum cohomology of $G/P$ where $P \subset G$ is a parabolic subgroup~\cite{LS10}. One of the primary goals of this paper is to provide a friendly guide to using the Peterson isomorphism in order to carry out calculations in quantum cohomology.  Since the reader familiar with the language of reductive groups over local fields and their associated affine Weyl groups can easily consult Sections 9 and 10 of \cite{LS10} for the original statement of the Peterson isomorphism, we elect to focus on the special case in which $G=SL_n$ and $P$ is a maximal parabolic subgroup to broaden accessibility.  

The majority of this section aims to reduce the statement of the Peterson isomorphism in full generality to the special case of a partial flag variety.  Section \ref{Sec:FullPeterson} concludes with a statement of the Peterson isomorphism which emphasizes how the $j$-basis for the affine Fomin-Stanley algebra can be used to do calculations in the quantum cohomology of the complete flag variety.  In Section \ref{Sec:PartialPeterson}, we state a version of the Peterson isomorphism for partial flag varieties, which we then specialize in Section \ref{PetIso} in order to explain how to use the $j$-basis to do quantum Schubert calculus for the Grassmannian.

\subsubsection{Root system of type $\tilde A_{n-1}$}\label{ars}

In order to state the Peterson isomorphism, we require some additional terminology associated to the underlying root system.  Since we are working in type $A$, the roots and coroots coincide, and so we phrase everything in terms of roots.  For an alternative similar treatment of the type $A$ root system preliminaries, we recommend Section 3 of \cite{BJV09}.  We refer the reader interested in other types to \cite{LS10} or \cite{kschur} for the more general development.  

 Let $\{\vec e_1,\dots,\vec e_n\}$ be the standard orthonormal basis of $\RR^n$.  The vectors $\alpha_i = \vec e_i -\vec e_{i+1}$ are called the {\bf simple roots} for type $A_{n-1}$. (Note that the Lie type $A_{n-1}$ should not be confused with the $(n-1)^{\text{st}}$ generator of the affine nilCoxeter algebra, which is also denoted $A_{n-1}$; it should always be clear from context which of the two is being discussed.)  The {\bf root lattice}, denoted by $Q$, is the $\ZZ$-span of the set $\Delta$ of simple roots.

Let $V \subsetneq \RR^n$ be the $(n-1)$-dimensional hyperplane of $\RR^n$ in which the $\alpha_i$ all reside. For a vector $v = c_1 \vec e_1 + ... + c_n \vec e_n \in \RR^n$ and a simple generator $s_i \in \tilde S_n$, we define 
\begin{equation}\label{E:S_nAction}
s_i\cdot v = 
\begin{cases}
c_1 \vec e_1 + \cdots + c_{i}\vec e_{i+1} + c_{i+1} \vec e_i + \cdots + c_n \vec e_n & \text{for}\ i \in \{1,2,...,n-1\}\\
(c_n+1) \vec e_1 + c_2 \vec e_2 + \cdots + c_{n-1} \vec e_{n-1} + (c_1-1) \vec e_{n} & \text{for}\ i=0.
\end{cases}
\end{equation}
For $w \in \tilde S_n$, choosing a reduced expression $w = s_{i_1}\cdots s_{i_r}$ and iterating this definition one generator at a time right to left provides an action of any $w \in \tilde S_n$ on any vector $v \in \RR^n$, which restricts to an action of $\tilde S_n$ on $V$.  Further restricting this action to one of $S_n$ on the simple roots $\Delta \subset V$ defines the set of all {\bf (finite) roots}
\begin{equation}
R = \{ w \cdot \alpha_i \mid w \in S_n, \alpha_i \in \Delta\}=\{\vec e_i-\vec e_j \mid i\ne j\}.
\end{equation}  
The set of {\bf positive roots} is defined as $R^+ = \{ \vec e_i - \vec e_j \mid i < j\}$.  The set of {\bf negative roots} is then defined to be $R^- = R\backslash R^+$. There is a distinguished root 
\begin{equation}
\theta = \sum_{i=1}^{n-1} \alpha_i = \vec e_1 - \vec e_n,
\end{equation}
 which is called the {\bf highest root}.

This action of the affine symmetric group on $\RR^n$ yields a realization of each element in $\tilde S_n$ as an affine transformation on $V$, which we now explain.  Define a hyperplane in $V$ indexed by a root $\alpha$ and an integer $p$ by
\begin{equation}
H_{\alpha,p} = \{ v \in V \mid \langle v, \alpha \rangle = p\},
\end{equation}
where $\langle \cdot, \cdot \rangle$ is the standard Euclidean inner product on $\RR^n$. We observe that the action of $s_i$ corresponds to a reflection across the hyperplane $H_{\alpha_i, 0}$, and acting by $s_0$ reflects across the hyperplane $H_{\theta,1}$. Defining
\[\mathcal H = \bigcup_{\alpha \in R^+, p\in \ZZ} H_{\alpha,p},\]
 the set
$V \backslash \mathcal H$ consists of disjoint convex regions called {\bf alcoves}. We will single out a particular alcove as the {\bf fundamental alcove}
\begin{equation}
\mathcal A_0 = \{ v \in V \mid 0 < \langle v, \alpha \rangle <1 \text{ for all $\alpha \in R^+$}\},
\end{equation}
which we can identify with the identity of the group $\tilde S_n$.

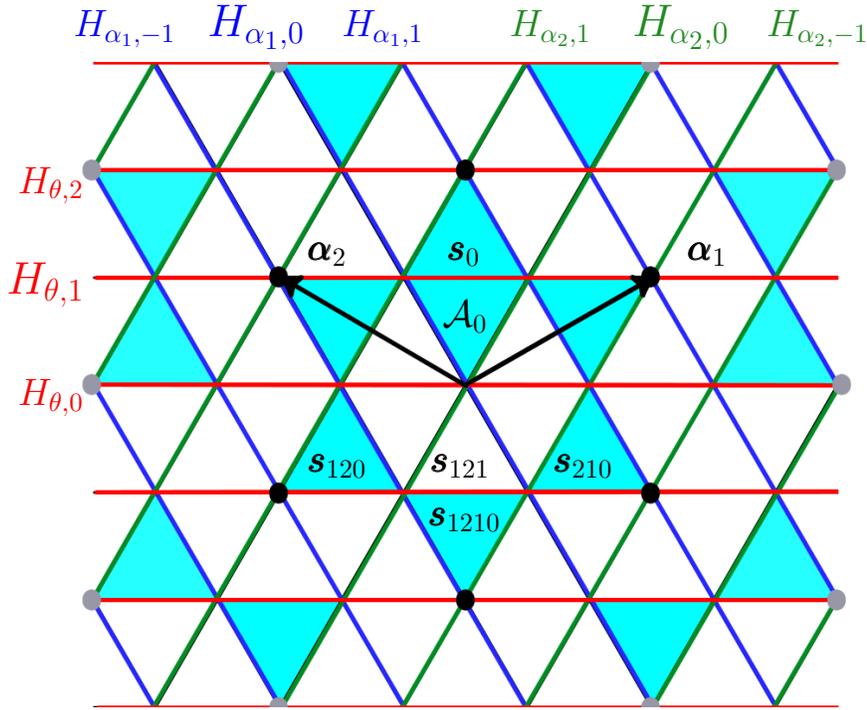
\begin{figure}[t!]
	\begin{centering}
\begin{overpic}[width=.9\columnwidth]{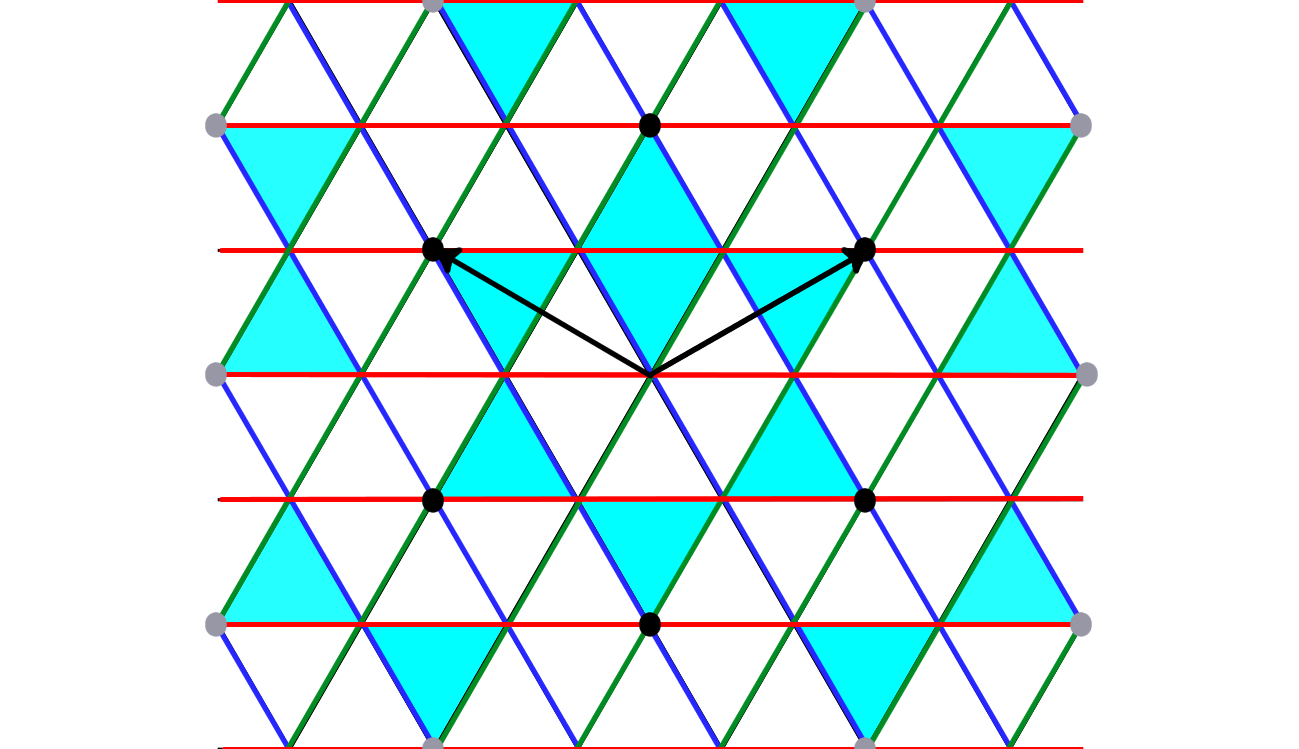}
\put(10,46){\Large $\textcolor{red}{H_{\theta,2}}$}
\put(9,37){\LARGE $\textcolor{red}{H_{\theta,1}}$}
\put(10,27){\Large $\textcolor{red}{H_{\theta,0}}$}
\put(48,34){\bf \Large $\mathcal A_0$}
\put(36,40){\bf \Large $\pmb \alpha_2$}
\put(70,40){\bf \Large $\pmb \alpha_1$}
\put(77,60){\Large $\textcolor[rgb]{.13,.54,.13}{H_{\alpha_2,-1}}$}
\put(65,60){\LARGE $\textcolor[rgb]{.13,.54,.13}{H_{\alpha_2,0}}$}
\put(54,60){\Large $\textcolor[rgb]{.13,.54,.13}{H_{\alpha_2,1}}$}
\put(15,60){\Large $\textcolor{blue}{H_{\alpha_1,-1}}$}
\put(27,60){\LARGE $\textcolor{blue}{H_{\alpha_1,0}}$}
\put(39,60){\Large $\textcolor{blue}{H_{\alpha_1,1}}$}
\put(48.5,40){\bf \Large$\pmb s_0$}
\put(36,21){\bf \Large$\pmb s_{120}$}
\put(47,21){\bf \Large$\pmb s_{121}$}
\put(58,21){\bf \Large$\pmb s_{210}$}
\put(46.8,16.2){\bf \Large$\pmb s_{1210}$}
\end{overpic}
\caption{The type $\tilde A_2$ root lattice, affine hyperplanes, and alcoves for $\tilde S_3$; see Example \ref{AlcoveEx}. }
\label{A2lattice}
\end{centering}
\end{figure}

Because acting on the roots by $s_i$ reflects over certain hyperplanes $H_{\alpha, p}$, we can think of the action of $s_i$, and consequently also of $w \in \tilde S_n$, as moving the fundamental alcove. Each alcove $\mathcal A$ then corresponds bijectively to the element $w \in \tilde S_n$ which maps $\mathcal A_0$ to
$\mathcal A$ when we reflect $\mathcal A_0$ over the sequence of hyperplanes associated to any reduced expression for $w$. Certain elements $w \in \tilde S_n$ simply translate $\mathcal A_0$ by a vector $\lambda \in Q$. We denote these {\bf translation elements} as $t_\lambda \in \tilde S_n$.  Note that this identification provides a natural embedding of the root lattice $Q \cong \{ t_\lambda \mid \lambda \in Q\}$ as a normal subgroup in $\tilde S_n$.  In fact, $\tilde S_n$ is the semi-direct product of this subgroup of translations with the finite symmetric group
\[ \tilde S_n \cong S_n \ltimes Q,\]
and we frequently use this decomposition to write $wt_\lambda \in \tilde S_n$, where $w \in S_n$ and $\lambda \in Q$.  Note that the action on roots corresponds to conjugation in the group, in the sense that $t_{w \cdot \alpha} =w t_{\alpha} w^{-1}$ for all $\alpha \in R$ and $w \in S_n$.

An element $\lambda \in Q$ is called {\bf dominant} if $\langle \lambda, \alpha_i \rangle >0$ for all $\alpha_i \in \Delta$, and we denote the set of all dominant elements by $Q^+$.  Similarly, an element $\lambda \in Q$ is called {\bf antidominant} if $\langle \lambda, \alpha_i \rangle < 0$ for all $\alpha_i \in \Delta$. The antidominant elements of the root lattice will play a special role in the Peterson isomorphism, and so we denote this set by $Q^-$. We refer to the set of translations $t_\lambda \in \tilde S_n$ where $\lambda \in Q^-$ as the {\bf antidominant translations}.

\begin{ex}\label{AlcoveEx}
In order to make this terminology more concrete, we draw and label the root lattice of type $\tilde A_2$ in Figure \ref{A2lattice}.  The finite roots corresponds to the vectors from the origin to the six black dots; other elements of the root lattice are indicated by gray dots.  We label the first few hyperplanes close to the origin; the family associated to the same positive root is drawn with the same color.  

The fundamental alcove $\mathcal A_0$ is based at the origin; reflecting across each of its faces corresponds to applying one of the three simple generators.  That is, the three generators $s_1, s_2,$ and $s_0$ act by reflecting over the three different colored hyperplanes $H_{\alpha_1,0}$, $H_{\alpha_2,0}$, and $H_{\theta,1}$ respectively, where $\theta = \alpha_1 + \alpha_2$ in this case. We label some alcoves with their corresponding element of $\tilde S_3$ using the shorthand $s_{i_1}s_{i_2}\cdots s_{i_r} = s_{i_1 i_2\cdots i_r}$. For example, the $s_{120}$ alcove is reached by successively reflecting the fundamental alcove $\mathcal A_0$ over the $H_{\theta,1}$, $H_{\alpha_2,0}$, and $H_{\alpha_1,0}$ hyperplanes in that order.   Each of the light blue triangles corresponds to a minimal length coset representative in $\tilde S_3/S_3$. 

The antidominant elements of the root lattice are those which lie ``below'' each of the hyperplanes $H_{\theta,0}$, $H_{\alpha_2,0}$, and $H_{\alpha_1,0}$ through the origin, where the fundamental alcove $\mathcal A_0$ is by definition ``above'' every hyperplane.  From this picture, for example, it is clear that the translation element $t_{-\theta}=t_{-\alpha_1 -\alpha_2}$ is the element $t_{-\theta} = s_{1210} \in \tilde S_3$, which is an antidominant translation.
\end{ex}


\subsubsection{The Peterson isomorphism for complete flags}\label{Sec:FullPeterson}

The {\bf complete flag variety} of type $A_{n-1}$ is the set of of all complete flags
\[ E = (\{0\} = E_0 \subset E_{1} \subset E_{2} \subset \cdots \subset E_{n-1} \subset E_n =  \CC^n), \]
where $\dim E_{i} = i$. This set is isomorphic to the quotient $SL_n(\CC)/B$, where $B$ is the subgroup of upper-triangular matrices.  Like the Grassmannian, the complete flag variety has the topological structure of a CW-complex given by the Schubert cells occurring in the Bruhat decomposition on $SL_n(\CC)$, and the Schubert cells are indexed by the permutations in $S_n$.  The cohomology $H^*(SL_n(\CC)/B)$ admits a $\ZZ$-basis of Schubert classes, and the quantum cohomology of the complete flag variety is a ring $QH^*(SL_n(\CC)/B)$ in which the Schubert classes form a $\ZZ[q_1, \dots, q_{n-1}]$-basis.  We can express any $\lambda \in Q^+$ in terms of the basis of simple roots as $\lambda = c_1\alpha_1 + \cdots + c_{n-1}\alpha_{n-1}$ for $c_i \in \ZZ_{\geq 0}$, and we denote by $q^\lambda = q_1^{c_1}\cdots q_{n-1}^{c_{n-1}}$ the degree corresponding to $\lambda$.  In quantum Schubert calculus for the complete flag variety, we seek formulas for expansions of all products 
\begin{equation}
\sigma_u*\sigma_v = \sum\limits_{\substack{w \in S_n, \\ \lambda \in Q^+}} c_{u,v}^{w,\lambda} q^\lambda \sigma_w,
\end{equation}
where $u, v \in S_n$ and $c_{u,v}^{w,\lambda} \in \ZZ_{\geq 0}$.  We remark that, in contrast to the situation for the quantum cohomology of the Grassmannian,  beyond special cases, no positive combinatorial formulas for the full flag three-point genus zero Gromow-Witten invariants $c_{u,v}^{w,\lambda}$ are known.

The Peterson isomorphism presents a method for doing quantum Schubert calculus using the combinatorics of the affine Fomin-Stanley algebra.  The Peterson isomorphism is between the homology of the affine Grassmannian and the quantum cohomology of the complete flag variety, appropriately localized.  On the side of $H_*(\mathcal{G}r)$, the multiplicatively closed set around which we localize is the set of homology classes for the antidominant translations.  On the quantum side, we must localize around each of the quantum parameters $q_1, \dots, q_{n-1}$.  We now state a simplified version of the original Peterson isomorphism which is best suited for using calculations on the affine side to do quantum Schubert calculus.

\begin{thm}[Peterson Isomorphism \cite{Pet}, \cite{LS10}] \label{FlagPeterson}
Suppose that $vt_\lambda \in \tilde{S}_n^0$, where $v \in S_n$ and $\lambda \in Q$. 
Then there is a $\ZZ$-algebra isomorphism
\begin{align*}
\Psi: H_*(\mathcal{G}r)[(j_{t_\mu}^0)^{-1} \mid \mu \in Q^-] & \longrightarrow QH^*(SL_n(\CC)/B)[q_i^{-1}]\\
j^0_{vt_\lambda} & \longmapsto q^{\lambda}\sigma_v.
\end{align*}
\end{thm}

 Note that if $\lambda \in Q^-$ is sufficiently antidominant, say whenever $\lambda \notin H_{\alpha_i,0}$ for any $\alpha_i \in \Delta$, then the element $vt_\lambda$ will automatically be a minimal length coset representative in $\tilde S_n/S_n$ for any $v \in S_n$.  In practice, the images of products of the elements $vt_{-\theta}$ where $\theta$ is the highest root completely determine the products of all of the Schubert classes in $QH^*(SL_n(\CC)/B)$.  Writing $\lambda = c_1\alpha_1 + \cdots + c_{n-1}\alpha_{n-1}$, if $\lambda \in Q^-$, then note that $c_i \in \ZZ_{\leq 0}$ for all $i \in \{1, \dots, n-1\}$.  Recalling the notation $q^\lambda = q_1^{c_1}\cdots q_{n-1}^{c_{n-1}}$, we see directly the need to localize around each of the quantum parameters, since all of the exponents arising in the image of the Peterson map will be nonpositive.

\begin{rmk}
From the algebraic and combinatorial standpoint, the fact that there is a non-equivariant version of the Peterson isomorphism as in Theorem \ref{FlagPeterson} is surprising.  Peterson's original proof, carried out in \cite{LS10}, verifies the equivariant quantum Chevalley-Monk formula on the images of the affine Schubert classes, which completely determines the product structure in $QH^*_T(G/P)$ by \cite{Mihalcea}.  Mihalcea's proof, however, relies on a double induction critically involving the degree of the polynomials in $\ZZ[\alpha_i]$ tracking the torus weights, which are simply unavailable in the non-equivariant setting.  This observation suggests an alternative explanation for Peterson's isomorphism, the proof of which would naturally leave a non-equivariant shadow.  There are many other perspectives on the Peterson isomorphism involving the Toda lattice, integrable systems, statistical mechanics, mirror symmetry, and other techniques in physics which may further explain the existence of this non-equivariant statement.
\end{rmk}

\subsubsection{Partial flag varieties}

In the remainder of this section, we adapt the treatment from \cite{LS10} in order to state the parabolic Peterson isomorphism for $QH^*(SL_n(\CC)/P)$, where $P \supseteq B$ is a standard parabolic subgroup; the reader interested in working with other Lie types is advised to directly reference Section 10 of \cite{LS10}.  In the case of $G=SL_n$, the quotient $G/P$ is a partial flag variety, which is indexed by a strictly increasing sequence ${\bf i} = (i_1, \dots, i_r) \in \NN^r$.  The {\bf partial flag variety for ${\bf i}$} is the set of all {\bf $r$-step flags}
\[ E = (\{0\} \subset E_{i_1} \subset E_{i_1} \subset \cdots \subset E_{i_r} \subset \CC^n), \]
where $\dim E_{i_j} = i_j$.  The two most important examples are the extreme cases of the complete flag variety $SL_n(\CC)/B$ with ${\bf i} = (1,2,\dots, n-1)$ studied in the previous subsection, and the Grassmannian with ${\bf i} = (m)$, which is the primary example we study in the next section.  Occasionally, it is useful for the flag vector ${\bf i}$ to indicate the existence of the zero subspace and the whole space as part of the flag, and so we define $i_0 = 0$ and $i_{r+1} = n$ by convention.

Each choice of $r$-step flag vector ${\bf i} = (i_1,...,i_r)$ for some $1 \le r \le n$ corresponds to a choice of standard parabolic subgroup $P \subset SL_n(\CC)$.  To be concrete, when $B$ is the subgroup of upper-triangular matrices, the parabolic subgroup $P$ is the subgroup of block upper-triangular matrices in which the blocks are of sizes $i_j -  i_{j-1}$ for $1 \leq j \leq r+1$ in that order.  Denote by $\Delta_P = \Delta \backslash \{ \alpha_{i_j} \mid 1 \leq j \leq r \}$ the simple roots corresponding to $P$.  For ease of notation, we occasionally refer exclusively to the indexing set $I_P = \{ m \in \NN \mid \alpha_m \in \Delta_P\}$, which is a subset of $I = \{ 1, \dots, n-1\}$. 

The parabolic subgroup $P$ also corresponds to a subset of finite roots
\begin{equation}\label{E:R_P}
R_P = \{  \vec e_{i} - \vec e_{j}  \mid i_{a-1} < i,j \le i_{a} \text{ for some } 1\leq a \leq r+1 \}.
\end{equation}
The roots in $R_P$ are precisely those finite roots which use only the elements $\alpha_{i_m} \in \Delta_P$ when they are expressed in terms of the basis $\Delta$ of simple roots.  By $R_P^+$ we denote the roots in $R_P$ which are also positive, and by $R_P^-$ those which are negative. The root lattice $Q_P$ associated to the parabolic $P$ is the $\ZZ$-span of the roots in $R_P$.  Define $\eta_P: Q \longrightarrow Q/Q_P$ to be the natural projection.  

The Weyl group $(S_n)_P$ for the Levi subgroup of $P$ is the product of the symmetric groups generated by the simple reflections $\{ s_{i_j} \mid 1 \leq j \leq r\}$ corresponding to $\Delta_P$.  That is, $(S_n)_P \cong S_{i_1} \times \cdots \times S_{i_j-i_{j-1}} \times \cdots \times S_{n-i_r}$.  The partial flag variety $SL_n(\CC)/P$ admits a decomposition into Schubert cells indexed by the elements of the quotient $S_n/(S_n)_P$.  We thus define $S_n^P$ to be the set of minimal length coset representatives in $S_n/(S_n)_P$.  As with the case of the complete flag variety in which case $I_P = \emptyset$, the Schubert classes indexed by $S_n^P$ form a $\ZZ[q_i : i \in I\backslash I_P]$-basis for $QH^*(SL_n(\CC)/P)$.  Writing $\lambda = \sum_{i \in I\backslash I_P} c_i\alpha_i \in Q/Q_P$, then we denote the corresponding quantum parameter by $q^\lambda = \prod_{i \in I \backslash I_P} q_i^{c_i}$.

\subsubsection{Affine roots in type $\tilde A_{n-1}$}

We now take the final steps towards stating the parabolic version of the Peterson isomorphism.  In the parabolic case, the isomorphism is between the homology of the affine Grassmannian modulo a certain ideal and the quantum cohomology of a partial flag variety, again suitably localized.  To define the ideal of affine Schubert classes with which we form the appropriate quotient, we require some additional terminology involving the affine roots in the type $\tilde A_{n-1}$ root system.

Define the {\bf null root} $\delta = \alpha_0 + \theta$ where $\alpha_0$ is formally added to our collection of roots $R$. We then define the set of {\bf affine roots} and {\bf positive/negative affine roots} as:
\begin{equation}
\begin{aligned}
R_{af} &= \{ \alpha + p \delta \mid \alpha \in R \text{ and } p \in \ZZ\}\\
R_{af}^+ &= \{ \alpha + p \delta \mid \alpha \in R \text{ and } p \in \ZZ_{>0}\} \cup R^+\\
R_{af}^- &= R_{af}\backslash R_{af}^+.
\end{aligned}
\end{equation}
We occasionally denote $\beta > 0$ and $\beta < 0$ for $ \beta \in R_{af}^+$ and $\beta \in R_{af}^-$, respectively.  Note that the positive affine roots $R^+_{af}$ are in natural bijection with the hyperplanes in $\mathcal H$. Denote the projection $\delta \mapsto 0$ on an affine root $\beta = \alpha+p\delta$ by $\overline{\beta} = \alpha \in R$.  Given a fixed parabolic subgroup $P$, we can also extend the definition of $R_P$ from \eqref{E:R_P} to the affine root system using this projection map
\begin{equation}
(R_P)^+_{af} = \{ \beta \in R_{af}^+ \mid \bar \beta \in R_P \}.
\end{equation}

We now define the \textbf{level-zero action} of $\tilde S_n$ on the affine roots
in $R_{af}$ by declaring that $s_j \cdot \delta = \delta$. Let the action of $s_j \in \tilde S_n$ on any simple root $\alpha_i \in \Delta$ and $\alpha_0 =\delta-\theta$ be defined by
\begin{equation}
s_j\cdot \alpha_i = \begin{cases} - \alpha_i & \text{if $i=j$} \\ \alpha_i + \alpha_j & \text{if $(i -j) \equiv \pm 1$ mod $n$} \\ 
\alpha_i & \text{otherwise}.
\end{cases}\label{sact}
\end{equation}
Since this action is linear, Equation \eqref{sact} determines an action $s_j \cdot \beta$ for all affine roots $\beta\in R_{af}$.  This action is then extended to all $w \in \tilde S_n$ in the same manner as for the finite roots by writing down a reduced expression $w = s_{i_1}\cdots s_{i_r}$ and applying the simple generators one at a time.  Note that the actions defined in \eqref{E:S_nAction} and \eqref{sact} agree whenever $i, j \neq 0$, but that the level-zero action by $s_0$ and/or on the affine root $\alpha_0$ defined in \eqref{sact} does not coincide with the action from \eqref{E:S_nAction}.  Since \eqref{sact} defines an action on $R_{af}$, whereas \eqref{E:S_nAction} defines an action on the real vector space $\RR^n$, it should always be clear from context which action is required.

The {\bf inversion set} of a word $w \in \tilde S_n$ is
defined in terms of this action of $w$ on the positive affine roots as follows:
\begin{equation}
\text{Inv}(w) = \{ \alpha \in R_{af}^+ \mid w \cdot \alpha \in R_{af}^-\}.
\end{equation}
It is a standard fact that the length of the word is equal to the number of its inversions:
\begin{equation}
\len(w) = |\inv(w)|.
\end{equation}

\subsubsection{The Peterson isomorphism for partial flags}\label{Sec:PartialPeterson}

In general, the homology of the affine Grassmannian $H_*(\mathcal{G}r)$ admits only a surjection onto the quantum cohomology of a partial flag variety.  In order to map isomorphically onto the smaller ring $QH^*(SL_n(\CC)/P)$, we will need to select certain affine Schubert classes to mod out by.  The following subset of $\tilde S_n$ characterizes the classes which will survive the quotient construction in the parabolic Peterson isomorphism
\begin{equation}
\tilde S_n^P = \{ w \in \tilde S_n \mid  w \cdot \beta \in R_{af}^+ \text{ for all } \beta \in (R_P)^+_{af} \}.
\end{equation}
 \begin{Def}
 Define an ideal in $H_*(\mathcal{G}r)$ by
\begin{equation}
J_P = \sum_{w \in \tilde S_n^0 \backslash \tilde S_n^P} \ZZ \ j^0_w,
\end{equation}
where recall that $\tilde S_n^0$ is the set of minimal length coset representatives in $\tilde S_n / S_n$.  
\end{Def}

The domain of the parabolic Peterson isomorphism will then be a suitably localized version of the quotient $H_*(\mathcal{G}r)/J_P$.  It thus remains only to specify the multiplicatively closed subset around which we localize on the affine side.  The following lemma permits a factorization of every element in $\tilde S_n$ into a product, the first component of which lies in this distinguished subgroup $\tilde S_n^P$, the second of which lies in the following subgroup
\begin{equation}
(\tilde S_n)_{P} = \{ v t_\lambda \in \tilde S_n \mid  \lambda \in Q_P,\ v \in (S_n)_P\}.
\end{equation} 

\begin{lemma}[Lemma 10.6 \cite{LS10}]  \label{Wfac}
For all $w \in \tilde S_n$, there is a unique factorization $w   = w_1 w_2$ where $w_1 \in \tilde S_n^P$ and $w_2 \in (\tilde S_n)_{P}$.
\end{lemma}

\begin{Def}
Define the map
\begin{equation}
\pi_P : \tilde S_n \to \tilde S_n^P \qquad \text{by} \qquad \pi_P(w) = w_1,
\end{equation}
where $w_1$ is from the unique factorization guaranteed by Lemma \ref{Wfac}.
\end{Def}

\begin{thm}[Parabolic Peterson Isomorphism \cite{Pet}, \cite{LS10}] \label{T:ParPeterson}
	Let $v \in S_n^P$, and suppose that $v\pi_P(t_\lambda) \in \tilde{S}^0_n$.  Then there is a $\ZZ$-algebra isomorphism
	\begin{equation}
	\begin{aligned}
	\Psi_P : ( H_*(\mathcal{G}r)/ J_P)[ (\overline{j}^0_{\pi_P(t_\mu)})^{-1} \mid \mu \in  Q^-]& \longrightarrow QH^*(SL_n(\CC)/P)[q_i^{-1} \mid i \in I \backslash I_P ] \\
	\overline{j}^0_{v \pi_P ( t_\lambda)}&\mapsto q^{\eta_P(\lambda)} \sigma_v,
	\end{aligned} 
	\end{equation}
	where $\overline{j}^0_w$ denotes the image of $j^0_w$ in the quotient $H_*(\mathcal{G}r)/ J_P$.
\end{thm}

\noindent Note that we use the same notation for the Schubert class in $QH^*(SL_n(\CC)/B)$ and $QH^*(SL_n(\CC)/P)$, and classes indexed by $\sigma_v$ with $v \in S_n^P$ do exist in both rings.  Since from now on we work exclusively in the case of the partial flag, it should always be clear from context which Schubert class is meant.


\section{The Parabolic Peterson Isomorphism for the Grassmannian}\label{PetIso}

In order to effectively use the parabolic Peterson isomorphism to perform calculations in quantum cohomology, we see from Theorem \ref{T:ParPeterson} that one needs to master the definition of $\tilde S_n^P$, which defines the ideal $J_P$, as well as the projection $\pi_P$ onto this subgroup.  All of the details are fully developed in this section in order to present an effective means for using the parabolic Peterson isomorphism to do quantum Schubert calculus in the Grassmannian $Gr(m,n)$.  The special case of the parabolic Peterson for the choice of ${\bf i} = (m)$ appears as Theorem \ref{GrPeterson}, together with several concrete examples.


\subsection{The ideal $J_P$ in the affine Fomin-Stanley algebra}

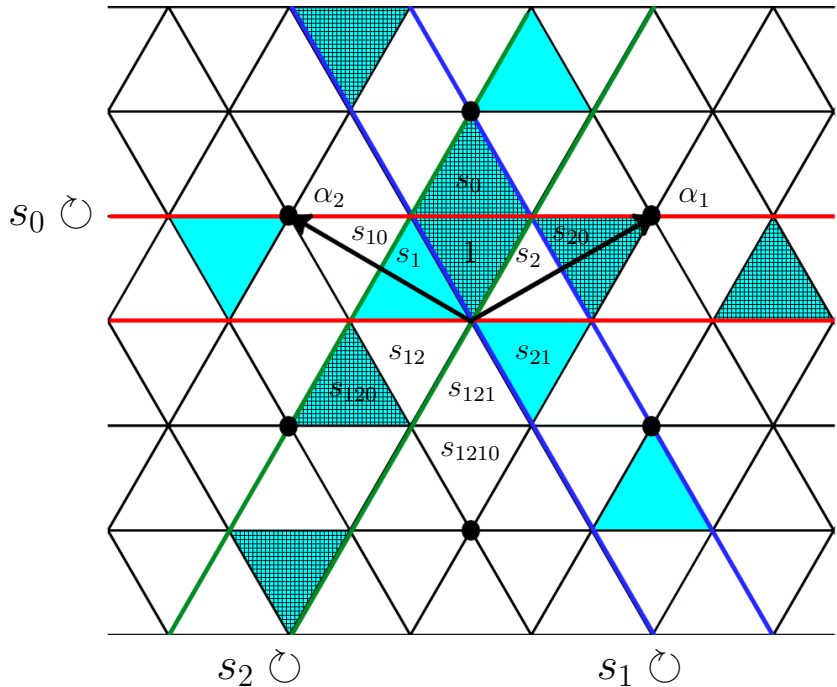
\begin{figure}[t]
	\begin{centering}
\begin{overpic}[width=0.6\linewidth]{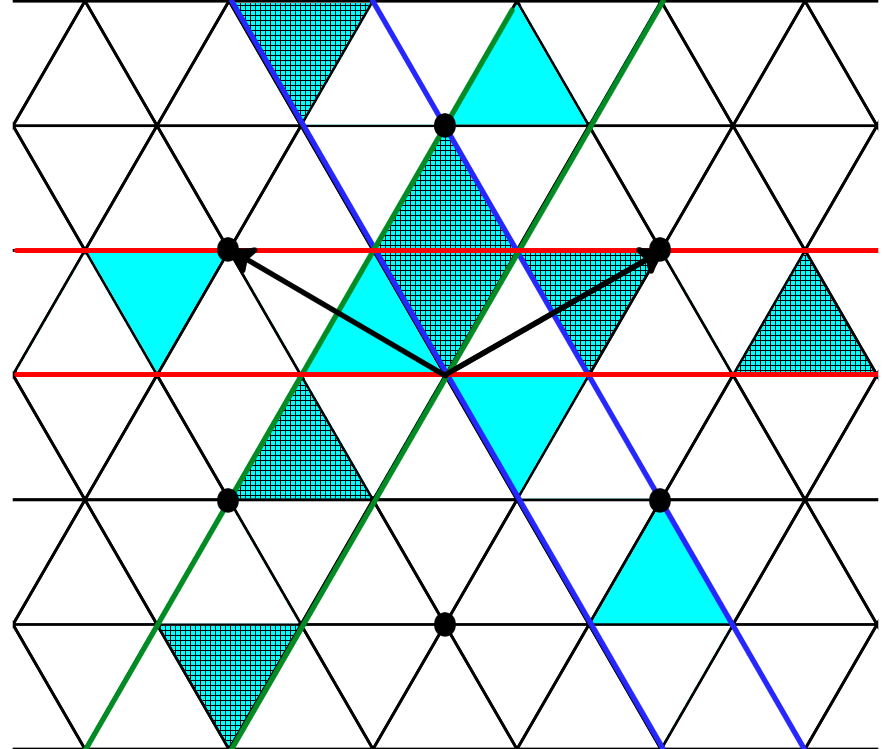}
	\put(49,50){\bf \large $1$}
	\put(40,50){\bf \large $s_1$}
	\put(56,50){\bf \large $s_2$}
	\put(39,37){\bf \large $s_{12}$}
	\put(47,32){\bf \large $s_{121}$}
	\put(46,24){\bf \large $s_{1210}$}
	\put(56,37){\bf \large $s_{21}$}
	\put(29,58){\bf \large $ \alpha_2$}
	\put(78,58){\bf \large $ \alpha_1$}
	\put(48,60){\bf \large $ s_0$}
	\put(-12,55){\bf \LARGE $s_0  \circlearrowright$}
	\put(67,-6){\bf \LARGE $ s_1  \circlearrowright$}
	\put(16,-6){\bf \LARGE $ s_2  \circlearrowright$}
	\put(34,53){\bf \large $ s_{10}$}
	\put(61,53){\bf \large $ s_{20}$}
	\put(31,32){\bf \large $ s_{120}$}
\end{overpic}
\vskip 30pt
\caption{The elements of $\tilde S_3^P$ for the parabolic $P$ corresponding to $Gr(1,3)$ are colored blue.  The shaded blue alcoves are the elements of $\tilde S_3^P \cap \tilde S_3^0$, which are also minimal length coset representatives.  }\label{Fig:J_PEx} 
\end{centering}
\end{figure}

Recall that the ideal with which we need to form a quotient of $H_*(\mathcal{G}r)$ is defined as \[J_P = \sum_{w \in \tilde S_n^0 \backslash \tilde S_n^P} \ZZ \ j^0_w\]
We thus first aim to understand the set $\tilde S_n^0 \backslash \tilde S_n^P$ to find which classes $j_w^0$ get killed in the quotient $H_*(\mathcal{G}r)/J_P$. To help the reader visualize this set, we provide an illustrated example showing $\tilde S_n^P$ in the case of $n=3$ in Figure \ref{Fig:J_PEx}.

In order to characterize the ideal $J_P$, we need to determine which elements $w \in \tilde S_n^0$ also appear in $\tilde S_n^P$.  Recalling the definition
\[  \tilde S_n^P = \{ w \in \tilde S_n \mid  w \cdot \beta \in R_{af}^+ \text{ for all } \beta \in (R_P)^+_{af} \},\]
we see that a priori, we need to check a positivity condition on the infinitely many affine roots in $(R_P)^+_{af}$.  The following lemma provides a significant reduction step to determining whether or not an element lies in $\tilde S_n^P$, enabling us to verify the positivity condition only on finitely many roots.  In particular, the lemma tells us that it is sufficient to consider only those finite roots $\alpha \in R_P^+$ when deciding whether or not an element $w \in \tilde S_n$ is in $\tilde S_n^P$.

\begin{lemma} A word $w \in\tilde S_n^P$ if and only if for all $\alpha \in R_P^+$, we have $w \cdot \alpha \in R^+ \cup (R^- + \delta)$.\label{small}
\end{lemma}
\begin{proof} We prove both statements by contrapositive.

($\Longleftarrow$)  Since $w \not \in \tilde S_n^P$, there exists $\beta' \in (R_P)^+_{af}$ such that
$w \cdot \beta' = \eta \in R^-_{af}$. By definition, $\beta' = \beta + n \delta$ with $n \ge 0$ and $\beta \in R_P^+ \cup (R_P^- + \delta)$. Then $w\cdot \beta = w\cdot (\beta'-n\delta) =\eta - n \delta \in R^-_{af}$, so there exists $\beta\in R_P^+ \cup(R_P^- + \delta)$ such that $w \cdot \beta\in R^-_{af}$. 
If $\beta \in R_P^+$, then we have found $w \cdot \beta \not \in R^+ \cup (R^- + \delta)$ as $R^-_{af}$ contains no positive roots. Otherwise, $\beta = -\alpha + \delta$ for some $\alpha \in R_P^+$. By definition of $R^-_{af}$, either
$w \cdot \beta  =  - \gamma - m \delta$ or $w\cdot \beta = \gamma - (m+1)\delta$ for some $\gamma \in R^+$ and $m\ge 0$, which implies that either $w \cdot \alpha = w \cdot(-\beta + \delta)
=\gamma + (m+1)\delta \not \in R^+ \cup (R^- + \delta)$ or $w \cdot \alpha = w \cdot(-\beta + \delta)
=-\gamma + (m+2)\delta\not \in R^+ \cup (R^- + \delta)$. In any case, we have found $\beta \in R_P^+$
or $\alpha = -\beta + \delta \in R_P^+$ such that either  $w\cdot \beta$ or $w \cdot \alpha$ is not in $R^+ \cup (R^- + \delta)$.

($\Longrightarrow$) Let $\alpha \in R_P^+$ with $w \cdot \alpha \not \in R^+ \cup (R^- + \delta)$. If $w\cdot \alpha = \beta - (n+1)\delta$ or $w \cdot \alpha = - \beta - n \delta$ for $\beta \in R^+$ and $n\ge 0$, then $w \not \in \tilde S_n^P$. Otherwise, either
$w \cdot \alpha = \beta + (n+1)\delta$ or $w \cdot \alpha = - \beta + (n+2) \delta$ for $\beta \in R^+$ and $n\ge 0$, whence
$w \cdot (-\alpha + \delta) = -\beta - n \delta$ or $ w \cdot (-\alpha + \delta) = \beta - (n+1) \delta$, so
$w \not \in \tilde S_n^P$.
\end{proof}

Using Lemma \ref{small}, we can now identify a large family of words which lie in $\tilde S_n^P$; namely, those which do not use $s_m$ in any reduced expression.

\begin{lemma} Let $I_P = I\backslash\{\kk\}$, and let $w=s_{i_1}\cdots s_{i_p} s_0 \in \tilde S_n^0$ be a reduced  expression with $i_\LL\ne \kk$ for all $\LL$. Then, $w \in \tilde S_n^P$.
\label{nosk}
\end{lemma}

\begin{proof}
Because $i_\LL \ne \kk$, the element $w$ corresponds to an $n$-core which fits in an $(n-\kk) \times \kk$ rectangle. The hook length of the upper-leftmost box is thus at most $n-1$, and so we know that the $n$-core coincides with the $k$-bounded partition. Suppose that the $k$-bounded partition has parts $\lambda = (\lambda_1,...,\lambda_{n-\kk})$.  Then for all $1 \leq i \leq n-m$, define a word $w_i$ as follows
\begin{equation}
w_i = s_{\lambda_i -i}s_{\lambda_i-2-i}\cdots s_{-i+2}s_{-i+1},
\end{equation}
with indices taken modulo $n$.  Define their product  
\begin{equation}
w = w_{\lambda_{n-\kk-1}}w_{\lambda_{n-\kk-1}}\cdots w_1,
\end{equation}
which is the word constructed from the $n$-core $\lambda$ using the algorithm discussed in Section \ref{Sec:AffSchClasses}. Here, since the shape fits inside the $(n-m)\times m$ rectangle, we have proceeded by removing only one box at a time, going row by row, removing rows from bottom to top.

Consider a root $\alpha = \vec e_{i} - \vec e_{j} \in R_P^+$ with $0 < i < j$ and $j \le k$. Then,
\begin{equation}
w_1 \cdot \alpha = \begin{cases} \vec e_{i} - \vec e_j & \text{if $\lambda_1< i$} \\
							\vec e_{i-1} - \vec e_{j} & \text{if $j\ge \lambda_1 \ge i$} \\
							\vec e_{i-1} - \vec e_{ \lambda_1-1} & \text{if $\lambda_1 > j$}
\end{cases}
\end{equation}
For each additional $w_i$ applied to $w_1 \cdot \alpha$, we know that, because $\lambda_a \ge \lambda_{a+1}$ for all $a$, the resulting root will always be of the form $\vec e_{x} - \vec e_y$ for some $i \geq x >m-n$ and $y = \max(\{j,\lambda_1-1\})$. Indeed, the maximum possible value for $\lambda_1$ is $\kk$. Furthermore, applying $w_i$ to $\vec e_x - \vec e_y$ will give either $\vec e_{x-1} - \vec e_y$ or $\vec e_x - \vec e_y$, so the minimum possible value for $x$ in $w \cdot \alpha = \vec e_x - \vec e_y$ is $i - n+\kk > \kk-n$. In any case, $w \cdot \alpha = \beta$ or $w\cdot \alpha = -\beta + \delta$ for some $\beta \in R^+$.  Indeed, if
$x > 0$, we have $w \cdot \alpha = \alpha_x + \cdots + \alpha_{y-1}$, and if
$x < 0$ then $\kk< x + n \le n-1,$ and we have $w \cdot \alpha = -\alpha_{y} - \alpha_{y+1} - \cdots - \alpha_{x+n-1} + \delta$ since $x + n - 1 > \kk -1$ and $y \le \kk-1$. Therefore, $w \cdot \alpha \in R^+ \cup (R^- + \delta)$.

Similarly, if  $\alpha  = \vec e_{i} - \vec e_{j} \in R_P^+$ with $ i < j\le n$ and $i \ge \kk+1$, we instead consider the transpose $\lambda^T$ having parts $\lambda_i^T$, and we write $w = w'_{\kk}w'_{\kk-1}\cdots w_1'$ where
\begin{equation}
w_i' = s_{i- \lambda^T_i}s_{i-\lambda^T_i+1}\cdots  s_{i-2}s_{i-1}.
\end{equation}
That is, we insert boxes column by column. The same argument then shows that $w \cdot \alpha \in R^+ \cup (R^- + \delta)$.  

Having checked all of the elements of $R_P^+$, by Lemma~\ref{small} we know that $w \in \tilde S_n^P$.
\end{proof}

In fact, Lemma \ref{nosk} is strong enough in the case of a maximal parabolic to yield a complete characterization of the words in $\tilde S_n^P$ whose corresponding $k$-bounded partitions are row or column shapes.  We formalize this observation afterwards in Remark \ref{hRmk}.

\begin{lemma}
Let $I_P = I\backslash\{\kk\}$. Then $w = s_{r-1,r-2,...,1,0} \in \tilde S_n^P$ if and only if $r \le \kk\le n-1$. Similarly, $w = s_{n-r+1,n-r+2,...,n-1,0} \in \tilde S_n^P$ if and only if $ r \le n-\kk \le n-1$. \label{hJP}
\end{lemma}
\begin{proof} ($\Longleftarrow$) In both cases this direction follows from Lemma~\ref{nosk} since $w$ does not use $s_\kk$.

($\Longrightarrow$) Let $n >r > \kk$. Consider the root $ \beta = \alpha_r + ... + \alpha_{n-1} = \vec e_{r} - \vec e_{n} \in (R_P)^+_{af}$. Then for $w = s_{r-1,r-2,...,1,0}$, we can calculate that
\begin{equation}
 w\cdot \beta = s_{r-1}\cdot( \alpha_r + \cdots + \alpha_{n-1} + \alpha_0 + \alpha_1 + \cdots + \alpha_{r-2}) = \alpha_{r-1}+ \delta.
\end{equation}
Therefore, $w \cdot (-\beta + \delta) \in R^-_{af}$, and so $w \not \in \tilde S_n^P$ by definition.

For $n > r > n-\kk$, we consider the root $\beta = \alpha_{1} + ... + \alpha_{n-r} =\vec e_{1} - \vec e_{n-r} \in (R_P)^+_{af}$. Then for $w = s_{n-r+1,n-r+2,...,n-1,0}$, we can calculate that
\begin{equation}
 w\cdot \beta = s_{n-r+1}\cdot(\alpha_1 + \cdots + \alpha_{n-r} + \alpha_{n-r+2} + \cdots + \alpha_{n-1}+\alpha_0) = \alpha_{n-r+1}+ \delta.
\end{equation}
Therefore, $w \cdot (-\beta + \delta) \in R^-_{af}$, and so $w \not \in \tilde S_n^P$ by definition.
\end{proof}

\begin{rmk}\label{hRmk}
Note that, because $r\le n-1$, then $w = s_{r-1,r-2,...,1,0}$ is a cyclically decreasing element in $\tilde S_n^0$, and so $j_w^0 = {\bf \tilde h}_r$.  Lemma \ref{hJP} then says that ${\bf \tilde h}_r \in J_P$ if and only if $r >\kk$. For the second family of elements $w' = s_{n-r+1,n-r+2,...,n-1,0}$ considered in Lemma \ref{hJP}, we then have ${\bf \tilde e}_r = j_{w'}^0$ and ${\bf \tilde e}_r \in J_P$ if and only if $r > n-\kk$.
\end{rmk}

The previous lemmas have certainly identified some of the elements of $\tilde S_n^P$ and thus some of the elements in $J_P$. The following lemma identifies a different family of elements in $J_P$; namely, those which support braid relations.

\begin{prop}
Let $I_P = I\backslash \{\kk\}$. If $w \in \tilde S_n^0$ and $w=w_1 s_is_{i+1}s_iw_2$ is a reduced expression for $w$ so that $\len(w) = \len(w_1) + \len(s_i s_{i+1}s_i) + \len(w_2)$, then $w \not \in \tilde S_n^P$. \label{braid}
\end{prop}

In order to prove Proposition \ref{braid}, we need two more technical facts. 

\begin{lemma}\label{braidtech1} Let $w=w_1w_2 \in \tilde S_n$ where $\len(w) = \len(w_1) +\len(w_2)$. If $w \in \tilde S_n^P$, then $w_2 \in \tilde S_n^P$. 
\end{lemma}
\begin{proof}
We argue by contradiction. Suppose there exists $\beta \in (R_P)_{af}^+$ such that $w_2 \cdot \beta \in R_{af}^-$. This means that $\beta \in \inv(w_2)$. Since $\beta \not \in \inv(w)$ by hypothesis, then after applying some $s_i$ which is a subexpression of $w_1$, the root $\beta$ will no longer be in the inversion set.  Therefore, since $\len(w) = |\inv(w)|$, the length of the expression has not increased with the addition of $s_i$, and so $w_1 w_2= w$ cannot satisfy $\len(w_1) + \len(w_2) =\len(w)$.
\end{proof}

Elements of a certain form will arise throughout the course of the next argument, and so we define the notation
\begin{equation}\label{sDsI}
s_{I,i,j}^p = s_{i,i+1,i+2,...,j-1,j} \quad \text{and} \quad s_{D,i,j}^p = s_{i, i-1,i-2,...,j+1,j}
\end{equation}
 to be a sequence of cyclically increasing, respectively decreasing, elements starting with $s_i$, ending with $s_j$, and containing $p$ copies of $s_0$.  For example, if $n=6$, then 
 \begin{equation}
 s_{I,1,3} = s_{1,2,3} \qquad\qquad s_{I,4,2}^1 = s_{4,5,6,0,1,2} \qquad\qquad s_{D,3,5}^2 = s_{3,2,1,0,6,5,4,3,2,1,0,6,5}. 
\end{equation}

\begin{lemma} Let $w \in \tilde S_n^0$. Suppose that  $w = w_1 s_{i}s_{i+1} s_{i} w_2$ is a reduced expression with $\len(w) = \len(w_1) + \len(s_i s_{i+1} s_i) +\len(w_2)$ for some $i \in \{0,1,..., n-1\}$.  If $i$ is chosen such that the braid $s_is_{i+1}s_i$ is as far right as possible in the reduced expression, then $w_2 = s_{i+2, i+3 ,..., n-1,i-1,i-2,...,1,0}$ and $i\not \in \{0,n-1\}$. \label{braidtech2}
\end{lemma}

\begin{proof}
Let $w_3 = s_i s_{i+1} s_i w_2$.
Note that $w_2$ must begin with either $s_{i+2}$ or $s_{i-1}$, since otherwise we could move the braid further to the right. Suppose $w_2$ begins with $s_{i+2}$. Because we cannot move the four terms $s_{i}s_{i+1}s_{i}s_{i+2}$ to the right, either $s_{i-1}$ appears to the right or $s_{i+3}$ does. Suppose that $s_{i-1}$ does not appear, so $s_{i+3}$ is to the right of $s_{i,i+1,i,i+2}$. Now assume that this same case occurs at each step where $s_{\rr+1}$ appears to the right of $s_\rr$. Ultimately, we will get the word
$w_3 = s_{i,i+1,i,i+2}s_{I,i+3,n-1}^\pp s_0$. If $\pp\ne 0$,  then we can slide $s_i$ to the right until we have $s_{i,i+1}s_{I,i+2,i-1}^1s_{i,i-1,i}s_{I,i+1,n-1}^{\pp-1}s_0$, which is a contradiction since this subexpression of $w$ contains a braid farther to the right than the original.
Therefore, $w_3 = s_{i,i+1,i,}s_{I,i+2,n-1}^0s_{0}$, but this is also a contradiction as the second $s_i$ commutes with everything to its right and $w \notin \tilde S_n^0$, unless $i=1$, $i=n-1$, or $i=0$. In the first case, $w_2$ has the claimed form. In the second case, $w_3 = s_{n-1,0,n-1,0} = s_{0,n-1}$ and so $w \not \in \tilde S^0_n$. In the third case, we slide $s_0$ to the right to get a braid $s_{0,n-1,0}$, which is another contradiction.  Therefore, $w_3 = s_{i,i+1,i}s_{I,i+2,\rr}^\pp s_{D,i-1,\mm}^{\qq}v s_0$ where $\pp\ge 0$, $\qq \ge 0$, and $v\in \tilde S_n$. 

Now consider $w_3$ expressed in this form.  We can always increase $\rr$ or decrease $\mm$ if there are still terms to the right of $s_{D,i-1,\mm}^\qq$ (besides $s_0$). Note that if we can slide $s_{D,i-1,\mm}^{\qq}$ to the right, then since we cannot slide $s_{I,i+2,\rr}^\pp$ to the right without shortening $w_2$, we must have
$w_3 = s_{i,i+1,i}s_{I,i+2,\rr}^\pp s_{\rr+1} s_{D,i-1,\mm}^{\qq}v' s_0$  for some new $v' \in \tilde S_n$ with $\len(v') = \len(v) - 1$ as any other term would cause a braid farther right (or would commute with everything). If we cannot move $s_{D,i-1,\mm}^{\qq}$ to the right, there must be an $s_{\mm-1}$ to its right by the same logic. Therefore, we increase the lengths of these chains at every step until we have
$w_3 = s_{i,i+1,i}s_{I,i+2,\rr'}^{\pp'} s_{D,i-1,\mm'}^{\qq'} s_0$ for some new $\pp', \qq', \rr',$ and $\mm'$.

If $\mm'\ne 1$ then $w \notin \tilde S_n^0$, and so $\mm'=1$. If $\qq' >0$, we know that by sliding $s_{\rr'}$ to the right, we will encounter the subexpression $s_{\rr',\rr'+1,\rr'}$, which gives us a braid further right than the original.  Therefore, we also have that $\qq'=0$. Similarly, $\rr' = n-1$, since otherwise we can slide $s_{\rr'}$ to the right until we have $s_{\rr',\rr'+1,\rr'}$ if $0 \le \rr'  < i+1$, or $w \not \in \tilde S_n^0$ as $s_{\rr'}$ commutes with everything to its right. Finally, $\pp'=0$ by the same argument as before.
Therefore, $w_3= s_{i,i+1,i}s_{I,i+2,n-1}^0s_{D,i-1,1}^0 s_0 = s_{i,i+1,i,i+2,...,n-1,i-1,...,1,0}$, and $w_2$ has the desired reduced expression. 

Finally, notice that, because $i+2$ and $i-1$ must
appear to the right of the braid, $i$ cannot be $0$ or $n-1$. If both $s_{I,i+2,n-1}^0$ and $s_{D,i-1,1}^0$ appear, we will have cancellation between the increasing and decreasing portions  (if $i=0$) or a braid appearing further right (if $i=n-1$). If only the increasing or decreasing portion appears instead, the expression will not be forced to end in $s_0$. Therefore, when the braid is farthest right, we know that $i\not \in \{0,n-1\}$.
\end{proof}

With these two lemmas in hand, we are now prepared to prove Proposition \ref{braid}.

\begin{proof}[Proof of Proposition~\ref{braid}]
By Lemma~\ref{braidtech2}, we can write $w = w_1s_{i}s_{i+1}s_{i} w_2$, where we have moved the braid as far right as possible, and so therefore $w_2 = s_{i+2,i+3...,n-1,i-1,i-2...,1,0}$. Consider the following pair of affine roots:
\begin{equation}
\beta_1 = \alpha_1 + \cdots + \alpha_i \quad \text{and} \quad \beta_{2} = \alpha_{i+1} + \cdots + \alpha_{n-1}.
\end{equation}
Since $I_P = I \backslash \{\kk\}$ we only remove one root to obtain $\Delta_P$, namely $\alpha_\kk$. Thus, one of the two roots $\beta_1$ or $\beta_2$ is guaranteed to be in $(R_P)^+_{af}$.

Now, calculate that $w_2 \cdot \beta_1 = -\alpha_{i+1}+ \delta$ and $w_2 \cdot \beta_2 = -\alpha_i + \delta$.  Therefore, when we act next by the braid $s_is_{i+1}s_i$, we obtain $s_{i}s_{i+1}s_i w_2 \cdot \beta_1 = \alpha_i + \delta$ and $s_{i}s_{i+1}s_iw_2 \cdot \beta_2 = \alpha_{i+1} + \delta$. Therefore, $s_{i}s_{i+1}s_i w_2 \cdot (-\beta_j + \delta) \in R_{af}^-$ for both $j=1,2$.   We have thus proved that there is a root $\beta \in (R_P)^+_{af}$ such that $s_{i}s_{i+1}s_{i}w_2 \cdot \beta \in R_{af}^-$.  By Lemma~\ref{braidtech1} we have that $w \not \in \tilde S_n^P$. 
\end{proof}


\subsection{The projection $\pi_P$ on the affine symmetric group}

Recall Lemma \ref{Wfac}, which says that all elements $w \in \tilde S_n$ admit a unique factorization 
\[ w   = w_1 w_2 \quad \text{where} \quad w_1 \in \tilde S_n^P \quad \text{and} \quad w_2 \in (\tilde S_n)_{P}.\]  
This lemma leads to the definition of the map $\pi_P$, which projects onto the first factor in this decomposition and plays the second critical role in the parabolic Peterson isomorphism:
\[
\pi_P : \tilde S_n \to \tilde S_n^P \qquad \text{by} \qquad \pi_P(w) = w_1.
\]
This section will thus be devoted to understanding the map $\pi_P$ in the case of $Gr(m,n)$.

We begin by recording several useful facts about how $\pi_P$ behaves on translation elements.

\begin{lemma}[Proposition 10.10 \cite{LS10}] Let $\lambda,\mu \in Q$, and let $\nu \in Q_P$.  Then,

1) $\pi_P(t_{\lambda + \mu}) = \pi_P(t_{\lambda})\pi_P(t_\mu)$

2) $\pi_P(t_{\lambda + \nu}) = \pi_P(t_{\lambda})$. \label{piPfacts}
\end{lemma}

Several detailed examples of calculations with $\pi_P$ applied to translation elements in which $P$ is a maximal parabolic subgroup are provided in groups of low rank in Example 10.9 of \cite{LS10}, and so we do not repeat any similar examples here.  We remark that generalizing those examples provided the inspiration for Proposition \ref{piP} below.

In practice, to use the parabolic Peterson isomorphism to do quantum Schubert calculus, it suffices to understand the image under $\pi_P$ of a single translation element, namely $\pi_P(t_{-\theta})$, where recall that the highest root is defined as $\theta = \sum \alpha_i = \vec e_1 - \vec e_n$. Therefore, we will really only need the following proposition to do computations in $QH^*(Gr(m,n))$.

\begin{prop}\label{piP} 
For $I_P= I\backslash\{\kk\}$, we have 
\begin{equation}
\pi_P(t_{-\theta}) = s_{\kk,\kk+1,...,n-1,\kk-1,...,1,0}.
\end{equation}
\label{piP}
\end{prop}

Before proving the proposition, we provide an example illustrating how it gets used in practice.

\begin{ex}
	We illustrate the element $\pi_P(t_{-\theta})$ for $P$ such that $SL_8(\CC)/P = Gr(3,8)$ below.  
	\begin{equation*}\small
	\tableau[Ys]{\textcolor{red}{\bullet} &\textcolor{red}{\bullet} & \textcolor{red}{\bullet}  \\
		\textcolor{red}{\bullet} &\textcolor{blue}{\bullet} & \textcolor{blue}{\bullet}  \\
		\textcolor{red}{\bullet} & \textcolor{blue}{\bullet}  \\
		\textcolor{red}{\bullet} & \textcolor{blue}{\bullet}  \\
		\textcolor{red}{\bullet} & \textcolor{blue}{\bullet}  \\
		\textcolor{red}{\bullet} & \textcolor{blue}{\bullet}  \\
		\textcolor{blue}{\bullet} & \textcolor{blue}{\bullet} 
	} 
	\end{equation*}
	The red dots indicate the boxes which belong to the 7-bounded partition corresponding to $u = \pi_P(t_{-\theta}) = s_{34567210}$, and the 7-bounded partition outlined by taking both the red and blue dots together corresponds to the element $u^2 = \pi_P(t_{-2\theta})$.  In general, the powers of the element $\pi_P(t_{-\theta})$ for $Gr(m,n)$ correspond to the $k$-bounded partitions obtained by adding successive $n$-rim hooks, each of which start in column 1 and end in column $m$; compare the construction of $\blacktriangle(\lambda)$ in \cite{LM2008}. See Example \ref{PetEx} for a more detailed discussion of the role of the powers of the element $\pi_P(t_{-\theta})$ in doing calculations in $QH^*(Gr(m,n))$.
\end{ex}

\begin{proof}[Proof of Proposition \ref{piP}]
Using Lemma~\ref{piPfacts}, since $\alpha_\kk-\theta \in Q_P$, we compute $\pi_P(t_{-\alpha_\kk})=\pi_P(t_{-\theta})$. 

It is easy to observe geometrically that
$t_{-\theta} = s_{\theta}s_0$, where $s_\theta = \tau_{1n}$ is the reflection in $S_n$ corresponding to $\theta$, which is the transposition that interchanges positions 1 and $n$.  Then, writing down a reduced expression for $\tau_{1n}$, we obtain
\[t_{-\theta}=\tau_{1n}s_0= s_{1,2...n-2,n-1,n-2,...2,1,0},\] 
 Note that $t_{-\alpha_\kk}$ can
be found using $t_{w \cdot \alpha} =w t_{\alpha} w^{-1}$. Letting 
\[v = s_{\kk+1,\kk+2,...,n-1,\kk-1,\kk-2,...,1},\]
it is straightforward to verify that 
\begin{equation}
v \cdot (-\theta) = v \cdot( -\alpha_1 - ... - \alpha_{n-1}) = -\alpha_\kk,
\end{equation}
and so $t_{-\alpha_\kk} = v t_{-\theta} v^{-1}$. We claim that $v t_{-\theta} =  s_{\kk,\kk+1,...,n-1,\kk-1,...,1,0}$. Using the notation of Equation \eqref{sDsI}, we can compute that 
\begin{equation}
\begin{aligned}
v t_{-\theta} &= v s_{I,1,n-2}^0s_{D,n-1,1}^0s_{0} \\
& = s^0_{I,\kk+1,n-1}s^0_{I,\kk,n-2}s^0_{D,n-1,1}s_{0}\\
& = s^0_{I,\kk+1,n-2}s^0_{I,\kk,n-3}s_{n-1,n-2,n-1,n-2}s^0_{D,n-3,1}s_{0} \\
& = s^0_{I,\kk+1,n-2}s^0_{I,\kk,n-3}s_{n-2,n-1} s^0_{D,n-3,1}s_{0}\\
& = s^0_{I,\kk+1,b+1}s^0_{I,\kk,n-1}s^0_{D,b,0} \\
& =s^0_{I,\kk+1,b} s^0_{I,\kk,b-1}s_{b+1,b,b+1,b} s^0_{I,b+2,n-1}s^0_{D,b-1,0}\\
& = s^0_{I,\kk+1,b}s^0_{I,\kk,n-1}s^0_{D,b-1,0}\qquad\qquad \text{(which we repeat until $b=\kk+1$)}\\
& = s_{\kk+1,\kk, \kk+1, \kk+2,...,n-1,\kk,\kk-1,...,1,0} \\
& = s_{\kk,\kk+1,...,n-1,\kk-1,\kk-2,...,1,0},
\end{aligned}
\end{equation}
as claimed.  The idea is that the decreasing part of $v$ cancels immediately, and, one-by-one, we can braid the terms in the increasing part of $v$ to get a cancellation.

Clearly $v^{-1} \in (S_n)_P \subset (\tilde S_n)_{P}$, and so if we can show that
$v t_{-\theta} \in \tilde S_n^P$, then we have factored $t_{-\alpha_\kk}$ as in Lemma~\ref{Wfac}. As in the proof of Lemma~\ref{nosk}, we consider $\beta = e_i-e_j =\alpha_{i} + \cdots + \alpha_{j-1}$ with $0<i<j$ and $j\le \kk$. Then,
\begin{equation}
s_{\kk,\kk+1,...,n-1,\kk-1,...,1,0} \cdot \beta = 
\begin{cases} 
\alpha_{i-1} + \cdots + \alpha_{j-2} & \text{ if $i>1$,} \\
 \alpha_{\kk} + \cdots + \alpha_{n-1} + \alpha_0 + \cdots + \alpha_{j-2} & \text{ if $i=1$},
\end{cases}
\end{equation}
since $j \le \kk$, we see that $vt_{-\theta} \cdot  \beta \in R^+ \cup (R^- + \delta)$. With a similar calculation, if instead we have
$ \beta = e_i-e_j$ with $i < j \le n$ and $i\ge \kk+1$, we  can show that
$vt_{-\theta} \cdot \beta \in R^+ \cup (R^- + \delta)$.  By Lemma~\ref{small}, we then know that $ v t_{-\theta} \in \tilde S_n^P$, and so $t_{-\alpha_m} = (vt_{-\theta})(v^{-1})$ is a factorization as in Lemma \ref{Wfac}.  Therefore, \[\pi_P(t_{-\theta})  = \pi_P(t_{-\alpha_m}) = v t_{-\theta} = s_{\kk,\kk+1,...,n-1,\kk-1,...,1,0},\] as desired.
\end{proof}

\subsection{The Peterson isomorphism for the Grassmannian}

We are now prepared to present the statement of the parabolic Peterson isomorphism for the special case of the Grassmannian $Gr(m,n)$.  Theorem \ref{GrPeterson} is an adaptation of the general parabolic Peterson isomorphism stated in \cite{Pet} and \cite{LS10} using the results of this section.  Following the statement, we provide examples based on our discussions of the ideal $J_P$ and the map $\pi_P$ which illustrate an effective means for computing products in $QH^*(Gr(m,n))$.

Recall that the map $\eta_P : Q \mapsto Q / Q_P$ is the natural surjection generated by sending $\alpha_i \mapsto 0$ for any $\alpha_i \in \Delta_P$.  In particular, for $I_P = I \backslash \{\kk\}$, this map preserves only $\alpha_\kk$.  We denote $q_m$ simply by $q$ and suppress the notation $\eta_P$, since there is only one quantum parameter in this case.

\begin{thm}[Peterson Isomorphism for the Grassmannian] \label{GrPeterson}
For $I_P = I \backslash \{\kk\}$, there is a $\ZZ$-algebra isomorphism
\begin{equation}
\begin{aligned}
\Psi_P : ( H_*(\mathcal{G}r)/ J_P)[ (\overline{j}_{u}^0)^{-1}]& \longrightarrow QH^*(Gr(\kk,n))[q^{-1} ] \\
\overline{j}_{v u^r}^0 &\longmapsto q^{-r} \sigma_v
\end{aligned} \label{PetEq}
\end{equation}
for $v \in S_n^P$ a minimal length coset representative and any $r \in \NN$ such that $vu^r \in \tilde S_n^0$.  Here, $u \in \tilde S_n^0$ is the hook shape corresponding to the element
\begin{equation}
u = s_{\kk,\kk+1,...,n-1,\kk-1,...,1,0}.
\end{equation} 
\end{thm}

Comparing the general statement in Thereom \ref{T:ParPeterson}, there we had localized around the set of classes corresponding to all antidominant translations.  By contrast, here we need only invert the single class $\overline{j}_{\pi_P(t_{-\theta})}^0$ since, for any antidominant $\lambda \in Q^-$, to compute $\pi_P(t_\lambda)$ we can start by
adding a vector  in $\alpha \in Q_P$ such that $\lambda + \alpha = - b\theta$, where $-b$ is exactly
the coefficient of $\alpha_\kk$ in the expansion of $\lambda$ in terms of the basis of simple roots. Then by Lemma~\ref{piPfacts}, we have that 
$\pi_P(t_\lambda) = \pi_P(t_{-b\theta}) = \pi_P(t_{-\theta})^b$, and so $\overline{j}_{\pi_P(t_{-\theta})}^0$ generates the set around which we localize on the affine side.

\begin{ex}\label{PetEx}
We illustrate how to use Theorem \ref{GrPeterson} and the related results of this section to produce the multiplication table for $QH^*(Gr(1,3))$. 

Of course, in this small example we can easily compute the three interesting products in this multiplication table using the quantum Pieri rule from Theorem \ref{qPieri}:
\begin{equation}\label{QHGr(13)}
\sigma_{\tiny \tableau[Ys]{ \ }}*\sigma_{\tiny \tableau[Ys]{ \ }} = \sigma_{\tiny \tableau[Ys]{ & \ }}
\qquad\qquad \sigma_{\tiny \tableau[Ys]{ \ }}*\sigma_{\tiny \tableau[Ys]{ & \ }} = q
\qquad\qquad \sigma_{\tiny \tableau[Ys]{& \ }}*\sigma_{\tiny \tableau[Ys]{ & \ }} = q \sigma_{\tiny \tableau[Ys]{\ }}.
\end{equation}
We now recreate these three quantum products using the parabolic Peterson isomorphism from Theorem \ref{GrPeterson}.  

First note that since ${\bf i} = (1)$ in the case of $Gr(1,3)$, then here $I_P = \{1, 2\} \backslash \{1\} = \{2\}$.  The elements of $S_3^P$ are those which are forced to end with an element of  $\langle s_1\rangle $ on the right, corresponding to each of the three Schubert classes in $QH^*(Gr(1,3))$:
\begin{equation}
\sigma_1 = \sigma_\emptyset \qquad\qquad \sigma_{s_1} = \sigma_{\tiny \tableau[Ys]{ \ }} \qquad\qquad  \sigma_{s_{21}} = \sigma_{\tiny \tableau[Ys]{ & \ }} 
\end{equation}

Now we turn to the map $\pi_P$.  From
Proposition~\ref{piP}, we have directly that $\pi_P(t_{-\theta})= u = s_{120}$. Using Equation~\eqref{PetEq} we can note the
following mappings 
\begin{equation}
\begin{aligned}
 \overline{j}^0_{u} & = \overline{j}^0_{s_{120}}  \longmapsto q^{-1}\sigma_\emptyset 
 \\
\overline{j}^0_{s_1 u} &= \overline{j}^0_{s_1 \cdot s_{120}}  = \overline{j}^0_{s_{20}}  \longmapsto q^{-1}\sigma_{\tiny \tableau[Ys]{ \ }}
 \\
\overline{j}^0_{s_{21} u}  & = \overline{j}^0_{s_{21} \cdot s_{120}} = \overline{j}^0_{s_0}  \longmapsto q^{-1}\sigma_{\tiny \tableau[Ys]{& \ }}.
\end{aligned}
\end{equation}
  (We remark that, in general, the element $vu^r \in \tilde S_n^0$ will be a minimal length coset representative in $\tilde S_n^0$ provided that $r \geq \text{diag}_m(\lambda(v))$.  Here, a single copy of $\pi_P(t_{-\theta})$ suffices, but already in $Gr(2,4)$ there are elements $v \in S_4^P$ which require the use of $u^2 = \pi_P(t_{-2\theta})$ in order to achieve $vu^2 \in \tilde S_4^0$.)
  
Returning to our example, the three quantum products in Equation \eqref{QHGr(13)} correspond via Theorem \ref{GrPeterson} to the following three products in the $j$-basis, respectively:
\begin{equation}
\begin{aligned}
 \overline{j}_{s_{20}}^0 \cdot \overline{j}_{s_{20}}^0 &\longmapsto (q^{-1}\sigma_{\tiny \tableau[Ys]{ \ }})*(q^{-1}\sigma_{\tiny \tableau[Ys]{ \ }}) \\
 \overline{j}_{s_{20}}^0 \cdot \overline{j}_{s_0}^0 &\longmapsto (q^{-1}\sigma_{\tiny \tableau[Ys]{ \ }})*(q^{-1}\sigma_{\tiny \tableau[Ys]{ & \ }}) \\
 \overline{j}_{s_{0}}^0 \cdot \overline{j}_{s_0}^0 &\longmapsto (q^{-1} \sigma_{\tiny \tableau[Ys]{& \ }})*(q^{-1}\sigma_{\tiny \tableau[Ys]{ & \ }}). 
 \end{aligned}
\end{equation}

Here, we are in the fortunate situation in which $j^0_{s_0} = {\bf \tilde h_1}$ and $j^0_{s_{20}} = {\bf \tilde e_2},$ and so we can use Definition \ref{hDef} to express these classes in terms of the basis $\{ A_w \mid w \in \tilde S_3\}$. We multiply directly in $\mathbb B$ to obtain:
\begin{equation}\label{jmultEx}
\begin{aligned}
 j_{s_{20}}^0 \cdot j_{s_{20}}^0 &= (\textcolor{red}{A_{20}}+A_{12} + A_{01} )^2 
= A_{2012}+ A_{1201} + \textcolor{red}{A_{0120}}  =  j_{s_{0120}}^0  \\
 j_{s_{20}}^0 \cdot j_{s_0}^0 &= (\textcolor{red}{A_{20}} + A_{12} + A_{01})(\textcolor{red}{A_{0}} + A_1 + A_2) \\&= A_{201} + A_{202}+ \textcolor{red}{A_{120}} + A_{121} + A_{010}  + A_{012}  = j_{s_{120}}^0 
\\
 j_{s_{0}}^0 \cdot j_{s_0}^0 &= (A_{0} + A_1 + A_2)^2 =
A_{01} + A_{02} + \textcolor{red}{A_{10}}+ A_{12} + \textcolor{red}{A_{20}} + A_{21}  = j_{s_{10}}^0 + j_{s_{20}}^0,
\end{aligned}
\end{equation}
where we have re-expressed the products in the $j$-basis using Theorem \ref{T:jBasisExpansion}, which says that the term $A_w$ for $w \in \tilde S_3^0$ (colored in red in the examples above) appears exactly once in the expansion of $j_w^0$ and not in any other $j_v^0$.

Next, recall from Lemma \ref{small} that we can easily verify whether an element $w \in \tilde S_3$ lies in $\tilde S_3^P$ by checking whether or not the two positive affine roots $\alpha_2$ and $-\alpha_2 + \delta =\alpha_0+ \alpha_1$ are inversions of $w$. In practice, as we see from Equation \eqref{jmultEx}, only finitely many elements of $\tilde S_3^P \cap \tilde S_3^0$ are required to recover all of the products in $QH^*(Gr(m,n))$, and these elements can always be chosen to be among the shortest in length. As seen in Figure \ref{Fig:J_PEx}, we have that
\begin{equation}
\{s_0,s_{20},s_{120},s_{0120},s_{20120},s_{120120}\} \subset \tilde S_3^P \cap \tilde S_3^0
\end{equation}
are the six shortest words in $\tilde S_3^P \cap \tilde S_3^0$, and that all other words in $\tilde S_3^0$ of length six or less lie outside of $\tilde S_3^P$.  The alcoves corresponding to these six words are illustrated in Figure \ref{Fig:J_PEx} at the beginning of the section.  Note that  $s_{10} \notin \tilde S_3^P$, which means that $ j^0_{s_{10}} \equiv 0 \mod J_P$.  The other three elements of the $j$-basis which appear in the products from \eqref{jmultEx} lie outside of $J_P$ and thus survive the quotient.    

We can now map our products in \eqref{jmultEx} via Theorem \ref{GrPeterson} to $QH^*(Gr(1,3))$ as follows:
\begin{equation}\label{jmultFinish}
\begin{aligned}
q^{-1} \sigma_{\tiny \tableau[Ys]{\ }} *q^{-1} \sigma_{\tiny \tableau[Ys]{\ }}  & \quad  \longleftrightarrow \quad  \overline{j}_{s_{20}}^0 \cdot \overline{j}_{s_{20}}^0 =   \overline{j}_{s_{0120}}^0 = \overline{j}^0_{s_{21} \cdot s_{120120}} = \overline{j}^0_{s_{21} \pi_P(t_{-2\theta})} \longmapsto  q^{-2} \sigma_{\tiny \tableau[Ys]{& \ }}  \\
q^{-1} \sigma_{\tiny \tableau[Ys]{\ }} *q^{-1} \sigma_{\tiny \tableau[Ys]{& \ }}  & \quad \longleftrightarrow \quad  \overline{j}_{s_{20}}^0 \cdot\overline{j}_{s_0}^0  = \overline{j}_{s_{120}}^0 \longmapsto q^{-1} \sigma_\emptyset
\\
q^{-1} \sigma_{\tiny \tableau[Ys]{& \ }} *q^{-1} \sigma_{\tiny \tableau[Ys]{& \ }}  & \quad \longleftrightarrow \quad  \overline{j}_{s_{0}}^0 \cdot \overline{j}_{s_0}^0 = \overline{j}_{s_{10}}^0 + \overline{j}_{s_{20}}^0 \longmapsto 0 + q^{-1}\sigma_{\tiny \tableau[Ys]{ \ }}
\end{aligned}
\end{equation}
Finally, multiplying everything through by $q^2$ to clear denominators, we exactly recover the three products in $QH^*(Gr(1,3))$ from \eqref{QHGr(13)}.
\end{ex}


\section{Posnikov's Approach and the Affine nilTemperley-Lieb Algebra}\label{Sec:nTL}

As an application of the parabolic Peterson isomorphism, in this section we recast Postnikov's affine approach to the quantum cohomology of the Grassmannian from \cite{P05} as a corollary of Theorem \ref{GrPeterson}.  Using the fact that a localization of $QH^*(Gr(m,n))$ yields a representation of the affine Lie algebra $\widehat{\mathfrak{sl}}_n$, Postnikov defines an action of the affine nilTemperley-Lieb algebra $\anTLn$ directly on quantum Schubert classes, and then uses this action to derive a quantum Pieri formula, a quantum Giambelli formula, and other classical results about $QH^*(Gr(m,n))$.  Moreover, Postnikov shows that a localized subalgebra of the affine nilTemperley-Lieb algebra generated by the ${\bf \tilde e}^n_r$ and ${\bf \tilde h}^n_r$ is isomorphic to the localization $QH^*(Gr(m,n))[q^{-1}]$, and that the generators ${\bf \tilde e}^n_r$ and ${\bf \tilde h}^n_r$ map to the Schubert classes $\sigma_{(1^r)}$ and $\sigma_{(r)}$, respectively, under this isomorphism.  In this section, we demonstrate that Postnikov's isomorphism is the composition of the parabolic Peterson isomorphism, followed by two duality isomorphisms which correct for the fact that the parabolic Peterson isomorphism itself does not map ${\bf \tilde h}_r \mapsto \sigma_{(r)}$ or ${\bf \tilde e}_r \mapsto \sigma_{(1^r)}$.

\subsection{The affine nilTemperley-Lieb algebra}

The affine nilTemperley-Lieb algebra is the quotient of the affine nilCoxeter algebra  formed by killing all braid relations.   Recalling from Proposition \ref{braid} that the non-commutative $k$-Schur functions indexed by affine symmetric group elements which support braid relations all lie in the ideal $J_P$, it is natural to expect that this quotient of $(\mathbb A_{af})_0$ relates to the quantum cohomology of the Grassmannian.

\begin{Def} The {\bf affine nilTemperley-Lieb algebra} $\anTLn$ is  is an associative $\ZZ$-algebra with $1$ generated by elements $a_i$ for $i\in \{0,1,...,n-1\}$ satisfying the relations:
\begin{equation}
\begin{aligned}
a_i^2&=0  \\
a_i a_j &= a_j a_i	&& \iff (i - j)\operatorname{mod} n \not \in \{-1, 1\} \\
a_i a_{i+1} a_i &= a_{i+1} a_i a_{i+1} = 0 &&\text{ for all $i$ with indices modulo n}.
\end{aligned}
\end{equation}
Note that the affine nilTemperley-Lieb algebra is simply a quotient of the affine nilCoxeter algebra by sending $A_i A_{i+1} A_i \mapsto 0$ for all $i$.  Otherwise, the image of $A_i$ is denoted by $a_i$ in the quotient.
\end{Def}

To distinguish the elements ${\bf \tilde h}_r, {\bf \tilde e}_r \in (\mathbb A_{af})_0$ from their images in this quotient, we denote by ${\bf \tilde h}^n_r$ and ${\bf \tilde e}^n_r$ the images of the non-commutative homogeneous and elementary symmetric functions in the affine nilTemperley-Lieb algebra. These elements satisfy certain relations in $\anTLn$ which mirror those satisfied by the classes $\sigma_{(r)}$ and $\sigma_{(1^r)}$ in $QH^*(Gr(m,n))$.

\begin{lemma}[Lemma 8.1 \cite{P05}]\label{techlem}
The elements ${\bf \tilde e}_r^n$ and ${\bf \tilde h}_r^n$ commute pairwise and satisfy
\begin{equation}\label{eheq}
\left(1 + \sum_{i=1}^{n-1} {\bf \tilde e}_i^n t^i\right)\left(1+\sum_{j=1}^{n-1}{\bf \tilde h}_j^n (-t)^j
\right)
= 1 + \left(\sum_{m=1}^{n-1} (-1)^{n-m} {\bf \tilde e}_m^n {\bf \tilde h}_{n-m}^n \right)t^n
\end{equation}
and ${\bf \tilde e}_i^n \cdot {\bf \tilde h}_j^n = 0$ if $i+j > n$.
\end{lemma}

Following \cite{P05}, we further define certain products of the elementary and homogeneous functions in the affine nilTemperley-Lieb algebra as follows:
\begin{equation}
{\bf z}_i = {\bf \tilde e}_i^n \cdot {\bf \tilde h}_{n-i}^n.
\end{equation}
Lemma 8.4 in  \cite{P05} shows that the elements ${\bf z}_i$ are in the center of $\anTLn$, permitting the following definition.

\begin{Def}
For any fixed $1\leq m \leq n-1$, we define the algebra
\begin{equation}
\anTLkn = \anTLn[{\bf z}_\kk^{-1}] / \langle {\bf z}_1, ..., {\bf z}_{\kk-1},{\bf z}_{\kk+1},...{\bf z}_{n-1}\rangle, \label{antlkn}
\end{equation}
and then let $\mathbb X \subset \anTLkn$ be the subalgebra generated by the ${\bf \tilde e}_r^n$ and/or ${\bf \tilde h}_r^n$ for all $1 \leq r \leq n-1$.
\end{Def}

\begin{rmk}
We note that in fact $\anTLkn = \anTLn[{\bf z}_\kk^{-1}]$. To see this, one can compute that ${\bf \tilde e}_j^n = {\bf \tilde h}_{n-i}^n = 0$ for $i < \kk$ and $j >\kk$ using only the fact that ${\bf z}_\kk$ is invertible. Therefore, in $\anTLn[{\bf z}_\kk^{-1}]$ we see that ${\bf z}_i = {\bf \tilde e}_i^n \cdot {\bf \tilde h}_{n-i}^n = 0 $ if $i < \kk$ or $i > \kk$, and so modding out by the elements ${\bf z}_i$ for $i \neq m$ in Eq.~\eqref{antlkn} is actually unnecessary.  The algebra $\anTLkn$ can thus be thought of most simply as $\anTLn$ localized around the ideal $\langle {\bf z}_\kk\rangle$.
\end{rmk}

Postnikov proceeds to prove that the subalgebra $\mathbb X$ generated by the ${\bf \tilde h}_r^n$ and/or ${\bf \tilde e}_r^n$  is isomorphic to the quantum cohomology of the Grassmannian, localized around the quantum parameter.

\begin{thm}[Proposition 8.5 \cite{P05}] 
For any fixed $1 \leq m \leq n-1$, the following map is an isomorphism of $\ZZ$-algebras which maps the generators as follows for all $1 \leq r \leq n-1$: \label{PostIso}
\begin{equation}
\begin{aligned}
\phi_{Po} :  \mathbb X \subset \anTLkn&\longrightarrow QH^*(Gr(\kk,n))[q^{-1}]\\
  {\bf \tilde h}_r^n  &\longmapsto \sigma_{(r)}\\
 {\bf \tilde e}_r^n &\longmapsto\sigma_{(1^r)} \\
{\bf z}_m &\longmapsto  q.
\end{aligned}
\end{equation} 
\end{thm}
\noindent Note that this result appears as Theorem 4.3 in the published version.

\subsection{Relationship to the parabolic Peterson isomorphism}

In the remainder of this section, we aim to show the following isomorphism naturally relating the subalgebra considered by Postnikov to the localized affine Fomin-Stanley algebra
\begin{equation}\label{PostIsoEq}
 (H_*(\mathcal{G}r)/ J_P)[(\overline{j}_{\pi_P(t_{-\theta})}^0)^{-1}] \cong \mathbb X \subset n\widehat {TL_{n-\kk,n}},
 \end{equation}
  where $I_P = I \backslash \{ m\}$.  Our conclusion is that Theorem \ref{PostIso} in \cite{P05} is really a direct consequence of the parabolic Peterson isomorphism.  We make these statements precise in Theorem \ref{algiso} and Corollary \ref{decomp}.

As a first step, we prove that the composition of the parabolic Peterson isomorphism $\Psi_P$, followed by the strange duality involution and the transpose isomorphism, map the generators in the same way as the map $\phi_{Po}$ from Theorem \ref{PostIso}.

\begin{thm} Let $I_P = I \backslash \{\kk\}$. The following composition is an isomorphism which maps the generators as follows for all $1 \leq r \leq n-1$:\label{modpet}
\begin{equation}
\begin{aligned}
T \circ \sd \circ \Psi_P : (H_*(\mathcal{G}r)/J_P)[(\overline{j}^0_{\pi_P(t_{-\theta})})^{-1}] &\longrightarrow QH^*(Gr(n-\kk,n))[q^{-1}]\\
{\bf \tilde h}_r &\longmapsto \sigma_{(r)} \\
{\bf \tilde e}_r &\longmapsto \sigma_{(1^r)}\\
{\bf \tilde z}_{n-m}={\bf \tilde e}_{n-m}\cdot {\bf \tilde h}_m &\longmapsto  q.
\end{aligned}
\end{equation}
\end{thm}
\noindent Note that this result appears as Theorem 4.7 in the published version.

\begin{proof}
This map is obviously an isomorphism since it combines the parabolic Peterson isomorphism with two duality isomorphisms.  In addition, recall from Definition \ref{noncomkschur} and Theorem \ref{jBasiskSchur} that for every $w \in \tilde S_n^0$, we have $j_{w}^0 = f({\bf \tilde h}_1,{\bf \tilde h}_2,...,{\bf \tilde h}_{n-1})$, and so it is sufficient to check that the generators are  mapped as claimed.

First note that  ${\bf \tilde h}_r = j_{w(r)}^0$, where $w(r) = s_{r-1, r-2, ..., 1, 0}$, which corresponds to a $k$-bounded shape given by a horizontal row of $r$ boxes since $r \leq n-1$. Now recall Lemma \ref{hJP}, which says that ${\bf \tilde h}_r \in \tilde S_n^P$ if and only if $r \leq m$.  Therefore, ${\bf \tilde h}_r \in J_P$ if and only if $m > r \geq n-1$, which means that $\Psi_P: {\bf \tilde h}_r \mapsto 0$ if the horizontal strip $(r)$ does not fit in $\mathcal{P}_{mn}$.  

Now for ${\bf \tilde h}_r$ with $1\le r \le \kk$, recalling the notation of Eq.~\eqref{sDsI}, we can clearly see that
\begin{equation}
w(r) = s_{r-1,...,0}= s^0_{I, r,\kk-1}s_{D,n-1,\kk}^0 s_{I,\kk,n-1}^0s_{D,\kk-1,1}^0s_{0} =
 s^0_{I, r,\kk-1}s_{D,n-1,\kk}^0 \pi_P(t_{-\theta})
\end{equation}
so using the statement of the Peterson isomorphism from Theorem \ref{GrPeterson}, we have
\begin{equation}
\Psi_P: \overline{j}^0_{w(r)} \longmapsto q^{-1}\sigma_\mu,
\end{equation}
where $\mu$ is a hook shape with width $n-\kk$ and height $\kk-r+1$. The main diagonal of $\mu$ has 1 box, the first row has full width in $\mathcal{P}_{n-m,n}$, and the first column is missing $r-1$ boxes. Therefore, when we map next by $\sd$, the shape will have one box in the first row and $r$ boxes in the first column; i.e. $\sd(\mu) = (1^r)$. Summarizing the composition, we have proved that
 \begin{equation}
{{\bf \tilde h}}_r \xmapsto{\Psi_P} q^{-1} \sigma_\mu \xmapsto{\ \sd\ } \sigma_{(1^{r})} \xmapsto{\ T\ } \sigma_{(r)}.
 \end{equation} 
 
Finally, since one can write $\sigma_{(1^r)}$ in terms of products and sums of $\sigma_{(r)}$, and since $s_{\lambda}^{(k)}  = s_\lambda$ if $\lambda \in \mathcal P_{mn}$, we know that ${\bf \tilde e}_r$ will be mapped as claimed. In addition, one can easily check that, modulo the ideal $J_P$, 
  \begin{equation}
  {\bf \tilde z}_{n-m} = {\bf \tilde e}_{n-m} \cdot {\bf \tilde h}_{m} = \overline{j}_{\pi_P(t_{-\theta})} \qquad \text{and} \qquad \sigma_{(1^{n-m})}\cdot \sigma_{(m)} = q
  \end{equation}
in $QH^*(Gr(n-\kk,\kk))$, as desired.
\end{proof}

Note the similarity of the statements in Theorem \ref{PostIso} and Theorem \ref{modpet}.  Our final goal will be to demonstrate that the two algebras in the domains are isomorphic, which shows that Postnikov's isomorphism from the localized affine nilTemperley-Lieb algebra is the exact same map that appears in Theorem \ref{modpet}.  We require one final technical lemma.

\begin{lemma}\label{wred} Let $I_P =I\backslash \{\kk\}$.  Suppose that $j_w^0 \in J_P$, where  $w = w_1 \pi_P(t_{-\theta})$ with $w_1 \in \tilde S_n^0$ and $\len(w) = \len(w_1) + \len(\pi_P(t_{-\theta}))$.
Then, 
\begin{equation}\label{inJP}
\overline{j}_{w_1}^0 \equiv 0 \in (H_*(\mathcal{G}r)/J_P)[(j^0_{\pi_P(t_{-\theta})})^{-1}].
\end{equation}
\end{lemma}

\begin{proof}
 Recall by Theorem \ref{T:jBasisExpansion} that
$j_v^0 = A_v + \sum\limits_{u \not \in \tilde S_n^0, \ \len(u)=\len(v)}c_{v}^u A_{u}$ for some $c_{v}^u \in \ZZ.$ 
Thus, 
\begin{align}
 j_{w_1}^0 \cdot j_{\pi_P(t_{-\theta})}^0 &=
\left( A_{w_1} + \sum_{\scriptsize\begin{array}{c}w' \not \in \tilde S_n^0 \\ \len(w') = \len(w_1) \end{array}}c_{w_1}^{w'} A_{w'}\right)
\left( A_{\pi_P(t_{-\theta})} + \sum_{\scriptsize\begin{array}{c}v' \not \in \tilde S_n^0 \\\len(v') = \len(\pi_P(t_{-\theta})) \end{array}}
c_{\pi_P(t_{-\theta})}^{v'} A_{v'}\right) \nonumber \\
&= A_{w_1 \pi_P(t_{-\theta})} + \sum_{w'} c_{w_1}^{w'}A_{w'\pi_P(t_{-\theta})} + \sum_{u \notin \tilde S_n^0} d_{u,w}A_u \nonumber \\
&= j_{w_1 \pi_P(t_{-\theta})}^0 + \sum_{\substack{w'\pi_P(t_{-\theta}) \in \tilde S_n^0\\ A_{w'\pi_P(t_{-\theta})}\ne 0}}c_{w_1}^{w'} j_{w'\pi_P(t_{-\theta})}^0, \label{jProduct}
\end{align}
where we use Theorem \ref{T:jBasisExpansion} in the last step to say that whenever $v \in \tilde S_n^0$, the factor $A_v$ occurs exactly once in the expansion of $j_v^0$ and not in any other element in the $j$-basis.  

Now, since $w' \notin \tilde S_n^0$, then each $w'$ appearing in the sum must end in a letter
$s_i$ with $i\ne 0$. Note that, for $i < \kk-1$, we have by Proposition \ref{piP} that
\begin{equation}
\begin{aligned}
s_i \pi_P(t_{-\theta}) &= s_i s_{\kk,\kk+1,...,n-1,\kk-1,...,1,0} \\
& = s_{\kk,\kk+1,...,n-1,\kk-1,...,i,i+1,i,...,1,0} \\
&= s_{\kk,\kk+1,...,n-1,\kk-1,...,i+1,i,i+1,...,1,0} \\
& = s_{\kk,\kk+1,...,n-1,\kk-1,...,i+1,i,...,1,0,i+1} \not\in \tilde S_n^0,
\end{aligned}
\end{equation}
which shows that $w'$ cannot end in $s_i$ with $0 < i < \kk-1$. A similar calculation shows $w'$ cannot
end in $i > \kk+1$. If $i = \kk$, then $A_{w' \pi_P(t_{-\theta})} = 0$ since we have two adjacent generators $A_m$.  Therefore, the only remaining possibilities are that $w'$ ends in either $s_{\kk-1}$ or $s_{\kk+1}$.  However, in either of these two cases, the element $w' \pi_P(t_{-\theta}))$ supports a braid relation.  Therefore, by Proposition~\ref{braid} and our hypotheses, 
 the sum in Eq.~\eqref{jProduct} lies in $J_P$. Finally, since $\overline{j}_{\pi_P(t_{-\theta})}^0$ is invertible, our claim follows.
\end{proof}

\begin{thm} \label{algiso}
Let $I_P = I\backslash\{\kk\}$. There is a $\ZZ$-algebra isomorphism
\begin{equation}\label{iota}
\iso: (H_*(\mathcal{G}r)/ J_P)[(\overline{j}_{\pi_P(t_{-\theta})}^0)^{-1}] \cong \mathbb X \subset n\widehat {TL_{n-\kk,n}},
\end{equation}
 where $\mathbb X$ is the subalgebra of the affine nilTemperley-Lieb algebra defined in Theorem~\ref{PostIso}. 
\end{thm}
\noindent Note that this result appears as Theorem 4.9 in the published version.

Of course, the existence of this isomorphism follows immediately because it can be expressed using Theorem \ref{PostIso} and Theorem \ref{modpet} as the following composition of isomorphisms
\begin{equation}
\iso = \phi_{Po}^{-1} \circ T \circ \sd \circ \Psi_P.
\end{equation}
However, our goal is to prove that this isomorphism arises independently of Postnikov's work, so we provide a proof which does not use on the map $\phi_{Po}$, but rather only the parabolic Peterson isomorphism.  Indeed, the natural projection $\mathbb B \longrightarrow \mathbb B/J_P$ induces an isomorphism between the localization of the affine Fomin-Stanley algebra and the subalgebra $\mathbb X$ of the affine nilTemperley-Lieb algebra considered by Postnikov.

\begin{proof}[Proof of Theorem \ref{algiso}]
Consider the natural surjection
\begin{equation}
p : H_*(\mathcal{G}r)
[({j}_{\pi_P(t_{-\theta})}^0)^{-1}] \longrightarrow n\widehat {TL_{n-\kk,n}}\qquad \text{by}\qquad {\bf \tilde h}_r \longmapsto {\bf \tilde h}_r^n.
\end{equation}
By a standard localization theorem, we have
\begin{equation}
(H_*(\mathcal{G}r)/  J_P)[(\overline{j}_{\pi_P(t_{-\theta})}^0)^{-1}] \cong H_*(\mathcal{G}r)[(j_{\pi_P(t_{-\theta})}^0)^{-1}]/J_P[(j_{\pi_P(t_{-\theta})}^0)^{-1}].
\end{equation}
 To prove the isomorphism in Equation \eqref{iota}, it thus suffices to show that 
 \begin{equation}
 \ker(p) =  J_P[(j_{\pi_P(t_{-\theta})}^0)^{-1}].
 \end{equation}

Since $p$ is the natural surjection, we know that $\ker(p)$ is the localization of the ideal that contains all elements $a$ satisfying the condition that for every monomial $A_w$ occurring in $a$, the element $w$ necessarily supports a braid relation by the definition of $n\widehat {TL_{n-\kk,n}}$. Since the element of the $j$-basis indexed by $w \in \tilde S_n^0$ at least contains the term $A_w$, then all the $j$-basis terms that appear in the expansion of $a$ support braids and are in $J_P$ by Proposition~\ref{braid}. Therefore, $\ker(p) \subseteq J_P[(j_{\pi_P(t_{-\theta})}^0)^{-1}]$ since localization respects inclusion.

Conversely, suppose that $j_w^0 \in J_P[(j_{\pi_P(t_{-\theta})}^0)^{-1}]$, and suppose by contradiction that $j_w^0 \not \in \ker(p)$. Further suppose that $\len(w)$ is the minimum length of any word satisfying this condition. By Lemma~\ref{nosk}, we know that $s_\kk$ appears in every reduced expression for $w$. Pick the rightmost instance of $s_\kk$ among all reduced expressions for $w$. There are clearly three cases: $w = w_1 s_{\kk,\kk-1,...,1,0}$ or $w = w_1 s_{\kk,\kk+1,...,n-1,0}$ or $w=w_1 s_{\kk,\kk+1,...,n-1,\kk-1,\kk-2,...,1,0}$ where in each case $\len(w) = \len(w_1) + \len(s_{\kk,...,0})$. 

Consider the first case $w = w_1 s_{\kk,\kk-1,...,1,0}$.  We first argue using the bijections discussed in Section \ref{Sec:AffSchClasses} that $w$ corresponds to a $\kconj$-bounded partition with at least $\kk+1$ boxes in the first row.  Let the $n$-core associated to $w$ be denoted $\mu=(\mu_1,...,\mu_p)$. Consider the $\kk+1$ rightmost boxes in the first row, starting in column $i=\mu_1 - \kk$. In order for one of these boxes to be removed when we biject to the $\kconj$-bounded partitions, the $i^{\text{th}}$ column must have $n-\kk+1$ boxes.  Thus, after inserting the first $\kk+1$ boxes, we must have inserted $i$ boxes in the $(n-\kk)^{\text{th}}$ row. The grid-numbers of this row begin with $\kk+1$, so inserting $i$ boxes into row $n-m$ also inserts $i$ boxes in the first row since there are shared labels. This means the first row has $i+\kk+1=\mu_1 +1$ boxes, which is a contradiction.  Therefore, $\lambda(w)$, the corresponding $k$-bounded partition, has at least $m+1$ boxes in the first row.  Now, by Lemma~\ref{hdecomp} and Definition \ref{noncomkschur}, we know that for some $c_{ij} \in \ZZ$, 
\begin{equation}\label{normal}
j_w^0 = {\bf s}_{\lambda(w)}^{(k)} = {\bf \tilde h}_{\lambda_1} {\bf s}_{(\lambda_2,...,\lambda_p)}^{(k)} + \sum_{i=1}^{n-1-\lambda_1} {\bf \tilde h}_{i+\lambda_1} \cdot \left(\sum_j c_{ij} {\bf s}_{\mu_{ij}}^{(k)}\right).
\end{equation}
By Lemma \ref{hJP}, we know that ${\bf \tilde h}_r \in J_P$ for $r > m$.  Since $p({\bf \tilde h}_r) = {\bf \tilde h}_r^n \equiv 0$ if $r > \kk$, then we also have that ${\bf \tilde h}_r \in \ker(p)$.  Since we have shown that $\lambda_1 > \kk$, we thus have $j_w^0 \in \ker(p)$.

Similarly, if $w =w_1 s_{\kk,\kk+1,...,n-1,0}$, then the $\kconj$-conjugate partition corresponds to the word
$w^{\omega_{\kconj}} = w_1^{\omega_{\kconj}} s_{n-\kk,n-\kk-1,...,1,0}$, since the $\kconj$-conjugate partition is
the transpose of the $n$-core, meaning that we reflect which diagonal we add the boxes to, as discussed in Section \ref{Sec:AffSchClasses}. Therefore, we can write
\begin{equation}\label{kconjcalc}
j_{w^{\omega_{\kconj}}}^0 = {\bf s}_{\lambda^{\omega_{\kconj}}}^{(k)} = {\bf \tilde h}_{\lambda_1^{\omega_{\kconj}}} {\bf s}_{(\lambda_2^{\omega_{\kconj}},...,\lambda_p^{\omega_{\kconj}})}^{(k)} + \sum_{i=1}^{n-1-\lambda_1^{\omega_{\kconj}}} {\bf \tilde h}_{i+\lambda_1^{\omega_{\kconj}}} \cdot \left(\sum_j c_{ij} {\bf s}_{\mu_{ij}^{\omega_{\kconj}}}^{(k)}\right)
\end{equation}
using the same $c_{ij}\in \ZZ$ as above.
 Recalling from Theorem \ref{kconjinv} that $\kconj$-conjugation is an automorphism on $H_*(\mathcal{G}r)$, we now $\kconj$-conjugate both sides of Equation \eqref{kconjcalc}, which expresses $j_w^0$ as a sum of terms ${\bf \tilde e}_i \cdot f_i$ for some elements $f_i \in \mathbb B$, where we know that $i > n-\kk$, since $\lambda_1^{\omega_{\kconj}} > n-\kk$.
By Lemma~\ref{hJP}, we know that ${\bf \tilde e}_r \in J_P$ for $r > n-m$, and since $p({\bf \tilde e}_r) = {\bf \tilde e}_r^n \equiv 0$ if $r > n-\kk$, then ${\bf \tilde e}_r \in \ker(p)$.  Since we have shown that $\lambda_1^{\omega_{\kconj}}  > n-\kk$, 
 we thus have that $j_w^0 \in \ker(p)$ in this case as well.

Finally, in the third case, we have $w = w_1 s_{\kk,\kk+1,...,n-1,\kk-1,...,1,0} = w_1 \pi_P(t_{-\theta})$ by Proposition \ref{piP}.  Therefore, by similar arguments to those made in the proof of Lemma~\ref{wred}, we see that $w_1$ must end in $s_0$, $s_{\kk-1}$, or $s_{\kk+1}$. In the latter case, we have that $w$ ends with
\[
s_{\kk+1} s_{\kk,\kk+1,...,n-1,\kk-1,...,1,0} = s_{\kk,\kk+1,\kk,\kk+2,...,n-1,\kk-1,...,1,0} = s_{\kk,\kk+1,\kk+2,...,n-1,\kk,\kk-1,...,1,0},
\]
which reduces to a previous case. This similarly holds if $w_1$ ends in $s_{\kk-1}$. When $w_1$ is forced to end in $s_0$, then by Lemma~\ref{wred}, we know that $j_{w_1}^0 \in J_P[(j^0_{\pi_P(t_{-\theta})})^{-1}]$. Therefore, since $\len(w_1) < \len(w)$, we know that $j_{w_1}^0 \in \ker(p)$ by definition of $w$. Recall from Eq.~\eqref{jProduct} in the proof of Lemma~\ref{wred} that 
\[j_{w_1}^0 j_{\pi_P(t_{-\theta})}^0 = j_{w}^0 +\sum_vc_{w_1}^{v} j_{v\pi_P(t_{-\theta})}^0,\] where $v$ is forced to end in either $s_{\kk-1}$
or $s_{\kk+1}$. This means that $j_{v \pi_P(t_{-\theta})}^0 \in \ker(p)$ because it reduces to a previous case argued above. Therefore, $j_w^0 \in \ker(p)$ in this third case as well.

We have reached a contradiction to the fact that $j_w^0 \notin \ker(p)$.  Therefore $J_P[j^0_{\pi_P(t_{-\theta})})^{-1}]\subseteq \ker(p)$, establishing that $\iso$ is an isomorphism.
\end{proof}

Theorem \ref{algiso} proves that the isomorphism established by Postnikov in \cite{P05} between the localizations of the subalgebra $\mathbb X$ of the affine nilTemperley-Lieb algebra and the quantum cohomology of the Grassmannian is really a direct consequence of the parabolic Peterson isomorphism.

\begin{cor} 
The map $\phi_{Po}$  from \cite{P05} defined in Theorem~\ref{PostIso} factors as
\begin{equation}
\phi_{Po} = T \circ \sd \circ \Psi_P \circ \iso^{-1}: \mathbb X \subset n\widehat{TL_{\kk n}}\longrightarrow QH^*(Gr(\kk,n))[q^{-1}],
\end{equation}
where $\iso$ is the isomorphism established in Theorem~\ref{algiso} and $I_P = I\backslash\{n-\kk\}$. \label{decomp}
\end{cor}
\noindent Note that this result appears as Corollary 4.10 in the published version.

\begin{proof} 
Having already established the isomorphism, it suffices to check that the image of this composition has the same effect on the generators as $\phi_{Po}$. By the proof of Theorem~\ref{algiso}, if $r \le n-\kk$, then ${\bf \tilde h}_r = \iso^{-1}({\bf \tilde h}_r^n)$, which maps to $\sigma_{(r)}$ in the codomain by Theorem~\ref{modpet}. Otherwise, ${\bf \tilde h}_r^n = {\bf \tilde h}_r = 0$. Since the ${\bf \tilde h}_r^n$ generate $\anTLkn$ and the $\sigma_{(r)}$ generate $QH^*(Gr(\kk,n))$, this concludes the proof.
\end{proof}

\begin{rmk}
We conclude with a remark about the prominent recurring role of hook shapes throughout this work.  Recall that the critical element $\pi_P(t_{-\theta})$ around which the affine Fomin-Stanley algebra is localized was a hook shape.  Moreover, powers of this element in turn become powers of the quantum parameter $q$ under the Peterson isomorphism, illuminating the rim hook rule of \cite{BCFF}.  Further, we have also seen that ${\bf \tilde h}_r$ was mapped under the Peterson isomorphism to a hook shape, and that the Strange Duality of Theorem~\ref{SD} connects rows and columns to certain other hook shapes.  Theorem \ref{modpet} further suggests that working with the often preferred bases of row and column shapes requires passing through the hook shapes as an intermediate step in order for the bases on the affine and quantum sides to match.  In other papers which explore the non-commutative $k$-Schur functions using the combinatorics of the affine nilCoxeter algebra, such as \cite{BSZ}, bases indexed by hook shapes often play a similarly prominent role, suggesting that such bases are perhaps more naturally suited to working directly with the Peterson isomorphism.
\end{rmk}

{\small
\bibliography{CookmeyerMilicevicRefs}
\bibliographystyle{alpha}}

\end{document}